\documentclass[12pt]{amsart}

\usepackage{amsmath,amssymb}
\usepackage[left=25mm,right=25mm,top=25mm,bottom=25mm]{geometry}
\usepackage{graphicx,psfrag,epstopdf}
\usepackage{ifthen}
\usepackage{verbatim}

\usepackage[dvips]{color}

\def\new#1{{\color{red}#1}}
\def\new#1{{#1}}

\def\dpr#1{#1}

\def\T{\mathbb T}
\def\A{\mathbb A}

\def\R{{\mathbb R}}
\def\N{{\mathbb N}}

\def\EE{{\mathcal E}}

\def\MM{{\mathcal M}}
\def\OO{{\mathcal O}}
\def\SS{{\mathcal S}}
\def\TT{{\mathcal T}}
\def\LL{{\mathcal L}}
\def\PP{{\mathcal P}}

\def\NN{{\mathcal N}}

\def\KK{\mathcal K}

\def\refine{{\tt refine}}
\def\RR{\mathcal R}
\def\res{{{\rm res}}}

\def\diam{{\rm diam}}
\def\norm#1#2{\|#1\|_{#2}}

\def\set#1#2{\big\{#1\,:\,#2\big\}}
\def\eps{\varepsilon}
\def\dual#1#2{\langle#1\,,\,#2\rangle}

\def\normHe#1#2{\norm{#1}{H^1(#2)}}
\def\normHeh#1#2{\norm{#1}{H^{1/2}(#2)}}
\def\normL2#1#2{\norm{#1}{L^2(#2)}}

\def\mu{\widetilde\varrho}

\newcounter{constantsnumber}
\def\setc#1{%
\ifthenelse{\equal{#1}{reliable}}{C_{\rm rel}}{%
\ifthenelse{\equal{#1}{efficient}}{C_{\rm eff}}{%
\ifthenelse{\equal{#1}{orthogonality}}{C_{\rm orth}}{%
\ifthenelse{\equal{#1}{dlr}}{C_{\rm dlr}}{%
\ifthenelse{\equal{#1}{scottzhang}}{C_{\rm sz}}{%
\ifthenelse{\equal{#1}{szoptimality}}{C_{\rm cea}}{%
\ifthenelse{\equal{#1}{refine}}{C_{\rm ref}}{%
\ifthenelse{\equal{#1}{nvb}}{C_{\rm nvb}}{%
\ifthenelse{\equal{#1}{afem}}{C_{\rm opt}}{%
\refstepcounter{constantsnumber}%
\label{const#1}C_{\theconstantsnumber}%
}}}}}}}}}}
\def\c#1{%
\ifthenelse{\equal{#1}{reliable}}{C_{\rm rel}}{%
\ifthenelse{\equal{#1}{efficient}}{C_{\rm eff}}{%
\ifthenelse{\equal{#1}{orthogonality}}{C_{\rm orth}}{%
\ifthenelse{\equal{#1}{dlr}}{C_{\rm dlr}}{%
\ifthenelse{\equal{#1}{scottzhang}}{C_{\rm sz}}{%
\ifthenelse{\equal{#1}{szoptimality}}{C_{\rm cea}}{%
\ifthenelse{\equal{#1}{refine}}{C_{\rm ref}}{%
\ifthenelse{\equal{#1}{nvb}}{C_{\rm nvb}}{%
\ifthenelse{\equal{#1}{afem}}{C_{\rm opt}}{%
C_{\ref{const#1}}%
}}}}}}}}}}

\def\osc#1{{\rm osc}_{\EE,#1}}
\def\oscT#1{{\rm osc}_{\TT,#1}}
\def\oscK#1{{\rm osc}_{\KK,#1}}
\def\oscD#1{{\rm osc}_{D,#1}}
\def\oscN#1{{\rm osc}_{N,#1}}
\def\wosc#1{\widetilde{{\rm osc}}_{\EE,#1}}

\def\conv{{\rm conv}}

\newtheorem{theorem}{Theorem}
\newtheorem{proposition}[theorem]{Proposition}
\newtheorem{lemma}[theorem]{Lemma}

\newtheorem{algorithm}[theorem]{Algorithm}

\newenvironment{remark}{\bigskip\noindent\textbf{Remark.}\ \it}{\qed\bigskip}
\newcounter{step}
\renewenvironment{proof}[1][]%
{\setcounter{step}{0}\medskip\noindent\ifthenelse{\equal{#1}{}}{\emph{Proof.~}}{\emph{#1.~}}}%
{\qed\bigskip}

\def\subsection#1
{
\bigskip

\refstepcounter{subsection}
{\noindent\bf\arabic{section}.\arabic{subsection}.~#1.~}
}

\begin{document}

\title[AFEM with Inhomogeneous Dirichlet Data]
{Convergence and Quasi-Optimality of Adaptive FEM\\
with Inhomogeneous Dirichlet Data}
\date{\today}

\author{M.~Feischl}
\author{M.~Page}
\author{D.~Praetorius}
\address{Institute for Analysis and Scientific Computing,
     Vienna University of Technology,
     Wiedner Hauptstra\ss{}e 8-10,
     A-1040 Wien, Austria}
\email{Michael.Feischl@tuwien.ac.at}
\email{Marcus.Page@tuwien.ac.at}
\email{Dirk.Praetorius@tuwien.ac.at\quad\rm(corresponding author)}



\begin{abstract}
We consider the solution of a second order elliptic PDE with inhomogeneous Dirichlet data by means of adaptive lowest-order FEM. 
As is usually done in practice, the given Dirichlet data are discretized by nodal interpolation. As model example serves the Poisson 
equation with mixed Dirichlet-Neumann boundary conditions. For error estimation, we use an edge-based residual error estimator which 
replaces the volume residual contributions by edge oscillations. For 2D, we prove convergence of the adaptive algorithm
even with quasi-optimal convergence rate. For 2D and 3D, we show convergence if the nodal interpolation operator is replaced by
the $L^2$-projection or the Scott-Zhang quasi-interpolation operator.
As a byproduct of the proof, we show that the Scott-Zhang operator converges pointwise to a limiting operator as the mesh is locally refined. This property might be of independent interest besides the current application. Finally,
numerical experiments conclude the work.
\end{abstract}


\maketitle

\section{Introduction}
\vspace*{-5mm}
\subsection{Model problem}
By now, the thorough mathematical understanding of convergence and quasi-optimality of $h$-adaptive FEM for second-order elliptic PDEs
has matured. However, the focus of the numerical analysis usually lies on model problems with homogeneous Dirichlet conditions,
i.e.\ $-\Delta u = f$ in $\Omega$ with $u=0$ on $\Gamma=\partial\Omega$,
see e.g.\ ~\cite{ckns,doerfler,ks,mns,stevenson}. \dpr{On a bounded Lipschitz domain in $\Omega\subset\R^2$ with polygonal boundary $\Gamma = \partial\Omega$, we consider}
\begin{align}\label{eq:strongform}
\begin{split}
 -\Delta u &= f \quad \text{in } \Omega,\\
 u &= g \quad \text{on } \Gamma_D,\\
 \partial_n u &= \phi \quad \text{on } \Gamma_N
\end{split}
\end{align}
with mixed Dirichlet-Neumann boundary conditions.
\dpr{The boundary $\Gamma$} is split into two relatively open boundary parts, namely the Dirichlet boundary $\Gamma_D$ and the Neumann boundary $\Gamma_N$, i.e.\ $\Gamma_D \cap \Gamma_N = \emptyset$ and $\overline\Gamma_D \cup \overline\Gamma_N = \Gamma$. \dpr{We assume the surface measure of the Dirichlet boundary to be positive $|\Gamma_D|>0$,}
whereas $\Gamma_N$ is allowed to be empty. The given data formally satisfy $f \in \widetilde H^{-1}(\Omega)$, $g \in H^{1/2}(\Gamma_D)$, and $\phi \in H^{-1/2}(\Gamma_N)$.
As is usually required to derive (localized) a~posteriori error estimators, we assume additional regularity of the given data, namely $f \in L^2(\Omega)$, $g \in H^1(\Gamma_D)$, and $\phi \in L^2(\Gamma_N)$.

\new{Whereas certain work on a posteriori error estimation for~\eqref{eq:strongform} has been done, cf.\ \cite{bcd,sv}, none of the proposed
adaptive algorithms have been proven to converge.}
\new{While} the inclusion of inhomogeneous Neumann conditions $\phi$ into the \new{convergence} analysis seems to be obvious, 
incorporating inhomogeneous Dirichlet conditions $g$ is technically more demanding 
and requires novel ideas. First, discrete finite element functions cannot satisfy general inhomogeneous Dirichlet conditions. Therefore, the adaptive algorithm has to deal with an additional discretization $g_\ell$
of $g$. Second, this additional error has to be controlled in the natural trace space which is
the fractional-order 
Sobolev space $H^{1/2}(\Gamma_D)$. Since the $H^{1/2}$-norm is non-local, the a~posteriori error analysis requires appropriate 
localization techniques. These have recently been developed in the context of adaptive boundary element 
methods~\cite{agp,cp:symm,cp:hypsing,effp,efgp,kp}:
Under certain orthogonality properties of $g-g_\ell\in H^1(\Gamma_D)$, the natural trace norm $\norm{g-g_\ell}{H^{1/2}(\Gamma_D)}$ is bounded by a locally weighted $H^1$-seminorm $\norm{h_\ell^{1/2}(g-g_\ell)'}{L^2(\Gamma_D)}$. Here, $h_\ell$ is the local mesh-width, and $(\cdot)'$ denotes the arclength derivative.
Finally, in contrast to homogeneous Dirichlet conditions $g=0$, we loose the Galerkin orthogonality in energy norm. This leads to certain technicalities to derive a contractive quasi-error which is equivalent to the overall Galerkin error in $H^1(\Omega)$.
In conclusion, quasi-optimality and even plain convergence of adaptive FEM with non-homogeneous Dirichlet data is a nontrivial
task. To the best of our knowledge, only~\cite{MNS03} analyzes convergence of adaptive FEM with inhomogeneous Dirichlet data. While the authors also consider the 2D model problem~\eqref{eq:strongform}
with $\Gamma_D=\Gamma$ and lowest-order elements, their analysis relies on an artificial non-standard marking criterion. 
Quasi-optimal convergence rates are not analyzed and can hardly be expected in general~\cite{ckns}.

It is well-known that the Poisson problem~\eqref{eq:strongform} admits a unique weak solution $u\in H^1(\Omega)$ with $u = g$ on $\Gamma_D$ in the sense of traces which solves the variational formulation
\begin{align}\label{eq:weakform}
 \dual{\nabla u}{\nabla v}_\Omega
 &= \dual{f}{v}_\Omega + \dual{\phi}{v}_{\Gamma_N}
 \quad \text{for all } v \in H^1_D(\Omega).
\end{align}
Here, the test space reads $H^1_D(\Omega) = \set{v\in H^1(\Omega)}{v = 0 \text{ on } \Gamma_D \text{ in the sense of traces}}$, and $\dual\cdot\cdot$ denotes the respective $L^2$-scalar products.

\subsection{Discretization}
For the Galerkin discretization, let $\TT_\ell$ be a regular triangulation of $\Omega$ into triangles $T\in\TT_\ell$. We use lowest-order conforming elements, where the ansatz space reads
\begin{align}
 \SS^1(\TT_\ell)
 = \set{V_\ell\in C(\overline\Omega)}{V_\ell|_T \text{ is affine for all }T\in\TT_\ell}.
\end{align}
Since a discrete function $U_\ell\in\SS^1(\TT_\ell)$ cannot satisfy \new{general} continuous Dirichlet conditions, we have to discretize the given data $g \in H^1(\Gamma_D)$. According to the Sobolev inequality on the 1D manifold $\Gamma_D$, the given Dirichlet data are continuous on $\overline\Gamma_D$. Therefore, the nodal interpoland $g_\ell$ of $g$ is well-defined. As is usually done in practice, we approximate $g\approx g_\ell$.
Again, it is well-known that there is a unique $U_\ell \in \SS^1(\TT_\ell)$ with $U_\ell = g_\ell$ on $\Gamma_D$ which solves the Galerkin formulation
\begin{align}\label{eq:galerkin}
 \dual{\nabla U_\ell}{\nabla V_\ell}_\Omega
 &= \dual{f}{V_\ell}_\Omega + \dual{\phi}{V_\ell}_{\Gamma_N}
 \quad\text{for all }V_\ell \in \SS^1_D(\TT_\ell).
\end{align}
Here, the test space is given by $\SS^1_D(\TT_\ell) = \SS^1(\TT_\ell) \cap H^1_D(\Omega) = \set{V_\ell \in \SS^1(\TT_\ell)}{V_\ell = 0 \text{ on } \Gamma_D}$.

\subsection{A~posteriori error estimation}
An element-based residual error estimator for this discretization reads
\begin{align}\label{eq1:estimator:T}
 \rho_\ell^2 = \sum_{T\in\TT_\ell}\rho_\ell(T)^2
\end{align}
with corresponding refinement indicators
\begin{align}\label{eq2:estimator:T}
 \begin{split}
 \rho_\ell(T)^2
 &:= |T|\,\norm{f}{L^2(T)}^2
 \\&\quad
 + |T|^{1/2}\big(\norm{[\partial_nU_\ell]}{L^2(\partial T\cap\Omega)}^2
 + \norm{\phi-\partial_nU_\ell}{L^2(\partial T\cap\Gamma_N)}^2
 + \norm{(g-g_\ell)'}{L^2(\partial T\cap\Gamma_D)}^2\big),
\end{split}
\end{align}
where $[\cdot]$ denotes the jump across edges.
We prove reliability and efficiency of $\rho_\ell$ (Proposition~\ref{prop:reliability:rho}) and discrete local reliability (Proposition~\ref{prop:dlr:rho}). Inspired by~\cite{pp}, we introduce an edge-based error estimator $\varrho_\ell$ which reads
\begin{align}\label{eq1:estimator:E}
 \varrho_\ell^2 = \sum_{E\in\EE_\ell}\varrho_\ell(E)^2.
\end{align}
For an edge $E\in\EE_\ell$, its local contributions read
\begin{align}\label{eq2:estimator:E}
 \varrho_\ell(E)^2 = \begin{cases}
 |E|\norm{[\partial_nU_\ell]}{L^2(E)}^2 + |\omega_{\ell,E}|\norm{f-f_{\omega_{\ell,E}}}{\omega_{\ell,E}}^2
 \quad&\text{if }E\subset\Omega,\\
 |E|\norm{\phi-\partial_nU_\ell}{L^2(E)}^2
 &\text{if }E\subseteq\Gamma_N,\\
 |E|\norm{(g-g_\ell)'}{L^2(E)}^2
 &\text{if }E\subseteq\Gamma_D.
 \end{cases}
\end{align}
Here, $\omega_{\ell,E}\subset\Omega$ denotes the edge patch, and $f_{\omega_{\ell,E}}$ denotes the corresponding integral mean.
The advantage of $\varrho_\ell$ is that the volume residual terms $|T|^{1/2}\norm{f}{L^2(T)}$ in~\eqref{eq2:estimator:T} are replaced by the edge oscillations $|\omega_{\ell,E}|^{1/2}\norm{f-f_{\omega_{\ell,E}}}{\omega_{\ell,E}}$, which are generically of higher order.
\new{The choice of $|E|\norm{(g-g_\ell)'}{L^2(E)}^2$ to measure the contribution of the Dirichlet data approximation is influenced by
the Dirichlet data oscillations, cf.\ Section~\ref{section:oscillations} below.}
We prove that
$\rho_\ell$ and $\varrho_\ell$ are locally equivalent (Lemma~\ref{lemma:local}) and thus obtain reliability and efficiency of $\varrho_\ell$ (Proposition~\ref{prop:reliability}) as well as discrete local reliability (Proposition~\ref{prop:dlr}).


\subsection{Adaptive algorithm}
We use the local contributions of $\varrho_\ell$ to mark edges for refinement in a realization (Algorithm~\ref{algorithm:doerfler})
of the standard adaptive loop (AFEM)
\begin{align}
 \boxed{\texttt{solve}}
 \quad\to\quad
 \boxed{\texttt{estimate}}
 \quad\to\quad
 \boxed{\texttt{mark}}
 \quad\to\quad
 \boxed{\texttt{refine}}
\end{align}
Our adaptive algorithm use variants of the the well-studied D\"orfler marking~\cite{doerfler} to mark certain edges for refinement. 
Throughout, we use newest vertex bisection, and at least marked edges are bisected. Given some initial mesh $\TT_0$, 
the algorithm generates successively locally refined meshes $\TT_\ell$ with corresponding discrete solutions $U_\ell\in\SS^1(\TT_\ell)$ of~\eqref{eq:galerkin}.

\subsection{Main results}
The first main result (Theorem~\ref{thm:contraction}) states that the \new{adaptive} algorithm leads to a contraction
\begin{align}
 \Delta_{\ell+1} \le \kappa\,\Delta_\ell
 \quad\text{for all }\ell\in\N_0
 \text{ and some constant }0<\kappa<1
\end{align}
for some quasi-error quantity $\Delta_\ell\simeq\varrho_\ell^2$ which is equivalent to the error estimator. In particular, this proves linear convergence of the adaptively generated solutions $U_\ell\in\SS^1(\TT_\ell)$ to the (unknown) weak solution $u\in H^1(\Omega)$ of~\eqref{eq:weakform}.
The main ingredients of the proof are an equivalent error estimator $\mu_\ell\simeq\varrho_\ell$ for which we prove some estimator reduction
\begin{align}
 \mu_{\ell+1}^{\,2} \le q\,\mu_\ell^{\,2} + C\,\norm{\nabla(U_{\ell+1}-U_\ell)}{L^2(\Omega)}^2
 \quad\text{for all }\ell\in\N_0
 \text{ and some }0<\kappa<1
 \text{ and }C>0,
\end{align}
see Lemma~\ref{lemma:reduction}, and a quasi-Galerkin orthogonality in Lemma~\ref{lemma:orthogonality}, whereas the general concept follows that of~\cite{ckns}.

The second main result is Theorem~\ref{thm:quasioptimality} which states that the outcome of the adaptive \new{algorithm} is 
quasi-optimal in the sense of Stevenson~\cite{stevenson}: Provided the given data 
$(f,g,\phi)\in L^2(\Omega)\times H^1(\Gamma_D)\times L^2(\Gamma_N)$ and the corresponding weak solution 
$u\in H^1(\Omega)$ of~\eqref{eq:weakform} belong to the approximation class
\begin{align}\label{eq:optimal:class}
 \A_s := \set{(u,f,g,\phi)}{\norm{(u,f,g,\phi)}{\A_s}:=\sup_{N\in\N}\big(N^s\sigma(N,u,f,g,\phi)\big)<\infty}
\end{align}
with
\begin{align}\label{eq:optimal:norm}
\begin{split}
 \sigma(N,u,f,g,\phi)^2 := \inf_{\TT_*\in\T_N} \Big\{&\inf_{W_*\in\SS^1(\TT_*)}
 \norm{\nabla(u-W_*)}{L^2(\Omega)}^2 
 +\new{\oscD{*}^2} 
 + \oscT{*}^2 + \oscN{*}^2\Big\},
\end{split}
\end{align}
the adaptively generated solutions also yield convergence order $\OO(N^{-s})$, i.e.
\begin{align}\label{eq:optimal:order}
 \norm{u-U_\ell}{H^1(\Omega)}
 \lesssim \big(\norm{\nabla(u-U_\ell)}{L^2(\Omega)}^2 + \new{\oscD{\ell}^2}\big)^{1/2}
 \lesssim (\#\TT_\ell-\#\TT_0)^{-s}.
\end{align}
Here, $\T_N$ denotes the set of all triangulations $\TT_*$ which can be obtained by local refinement of the initial mesh 
$\TT_0$ such that $\#\TT_*-\#\TT_0\le N$. Moreover, $\oscT{*}, \new{\oscD{*}}$, and $\oscN{*}$ denote the data oscillations of the volume data $f$,
\new{the Dirichlet data $g$}, and 
the Neumann data $\phi$, see Section~\ref{section:oscillations}. 

The ingredients for the proof are the observation that the proposed marking \new{strategy is} optimal \new{(Proposition~\ref{prop:doerfler})} and the C\'ea-type estimate
\begin{align}
\begin{split}
&\norm{\nabla(u-U_\ell)}{L^2(\Omega)}^2 + \new{\oscD{\ell}^2} 
 \le 
\c{szoptimality}\big(\inf_{W_\ell \in \SS^1(\TT_\ell)}
\norm{\nabla(u-W_\ell)}{L^2(\Omega)}^2 \!+\! \new{\oscD{\ell}^2} \big)
\end{split}
\end{align}
for the Galerkin solution $U_\ell\in\SS^1(\TT_\ell)$ in Lemma~\ref{lem:quasiopt}.

\new{For 3D, nodal interpolation of the Dirichlet data $g\in H^1(\Gamma)$ is not well-defined. In the literature, it is proposed to discretize $g$ by use of the 
$L^2$-projection~\cite{bcd} or the Scott-Zhang projection~\cite{sv}. Our third theorem (Theorem~\ref{thm:3Dconvergence}) states convergence of the 
adaptive algorithm for either choice in 2D as well as 3D. 
\dpr{The proof relies on the analytical observation that, under adaptive mesh-refinement, the Scott-Zhang projection converges pointwise to a limiting operator (Lemma~\ref{lem:apriori:sz}), which
might be of independent interest.}
Finally, we stress that the same results (Thm.~\ref{thm:contraction},~\ref{thm:quasioptimality},~\ref{thm:3Dconvergence}) hold if the element-based estimator
$\rho_\ell$ from~\eqref{eq1:estimator:T}--\eqref{eq2:estimator:T} instead of the edge-based estimator $\varrho_\ell$ is used and if Algorithm~\ref{algorithm:doerfler}
marks certain elements for refinement.}

\subsection{Outline}
The remainder of this paper is organized as follows: We first collect some necessary preliminaries on, e.g., newest vertex bisection 
(Section~\ref{section:nvb}) and the Scott-Zhang quasi-interpolation operator (Section~\ref{section:sz}). 
Section~\ref{section:aposteriori} contains the analysis of the a~posteriori error estimators $\rho_\ell$ 
from~\eqref{eq1:estimator:T}--\eqref{eq2:estimator:T} and $\varrho_\ell$ from~\eqref{eq1:estimator:E}--\eqref{eq2:estimator:E}. 
Moreover, we state \new{the adaptive Algorithm} in Section~\ref{section:doerfler}. The 
convergence is \new{shown} in Section~\ref{section:convergence}, while the quasi-optimality results are found in 
Section~\ref{section:optimal}. Whereas the major part of the paper is concerned with the 2D model problem,
Section~\ref{section:remarks3d} considers convergence of AFEM for 3D. Finally, some numerical experiments conclude the work. 

\section{Preliminaries}

\vspace*{-3mm}
\subsection{Notation}\label{sec:notation}
Throughout, $\TT_\ell$ denotes a regular triangulation which is obtained by $\ell$ steps of (local) newest vertex bisection for a given initial triangulation $\TT_0$. By $\KK_\ell:=\KK_\ell^\Omega\cup \KK_\ell^\Gamma$, we denote the set of all interior nodes, respectively the set of all boundary nodes of $\TT_\ell$. By $\EE_\ell$, we denote the set of all edges of $\TT_\ell$ which is split into the interior edges $\EE_\ell^\Omega = \set{E\in\EE_\ell}{E\cap\Omega\neq\emptyset}$ and boundary edges $\EE_\ell^\Gamma = \EE_\ell\backslash\EE_\ell^\Omega$. We restrict ourselves to meshes $\TT_\ell$ such that each $T\in \TT_\ell$ has an interior node, i.e.\ $\partial T \cap \KK_\ell^\Omega \neq \emptyset$. Note, that this is only an assumption on the initial mesh $\TT_0$.  We assume that the partition of $\Gamma$ into Dirichlet boundary $\Gamma_D$ and Neumann boundary $\Gamma_N$ is resolved, i.e.\ $\EE_\ell^\Gamma$ is split into $\EE_\ell^D=\set{E\in\EE_\ell}{E\subseteq\overline\Gamma_D}$ and
$\EE_\ell^N=\set{E\in\EE_\ell}{E\subseteq\overline\Gamma_N}$. Note that $\EE_\ell^D$ (resp.\ $\EE_\ell^N$) provides a partition of 
$\Gamma_D$ (resp.\ $\Gamma_N$).

For a node $z\in\KK_\ell$, the corresponding patch is defined by
\begin{align}\label{eq:patch:z}
 \omega_{\ell,z} = \bigcup\set{T\in\TT_\ell}{z\in\partial T}.
\end{align}
For an edge $E\in\EE_\ell$, the edge patch is defined by
\begin{align}\label{eq:patch:E}
 \omega_{\ell,E} = \bigcup\set{T\in\TT_\ell}{E\subset\partial T}.
\end{align}
Moreover, for a given node $z\in\KK_\ell$,
\begin{align}
 \EE_{\ell,z} = \bigcup\set{E\in\EE_\ell}{z\in E}
\end{align}
denotes the star of edges originating at $z$.

\begin{figure}[t]
\centering
\psfrag{T0}{}
\psfrag{T1}{}
\psfrag{T2}{}
\psfrag{T3}{}
\psfrag{T4}{}
\psfrag{T12}{}
\psfrag{T34}{}
\includegraphics[width=35mm]{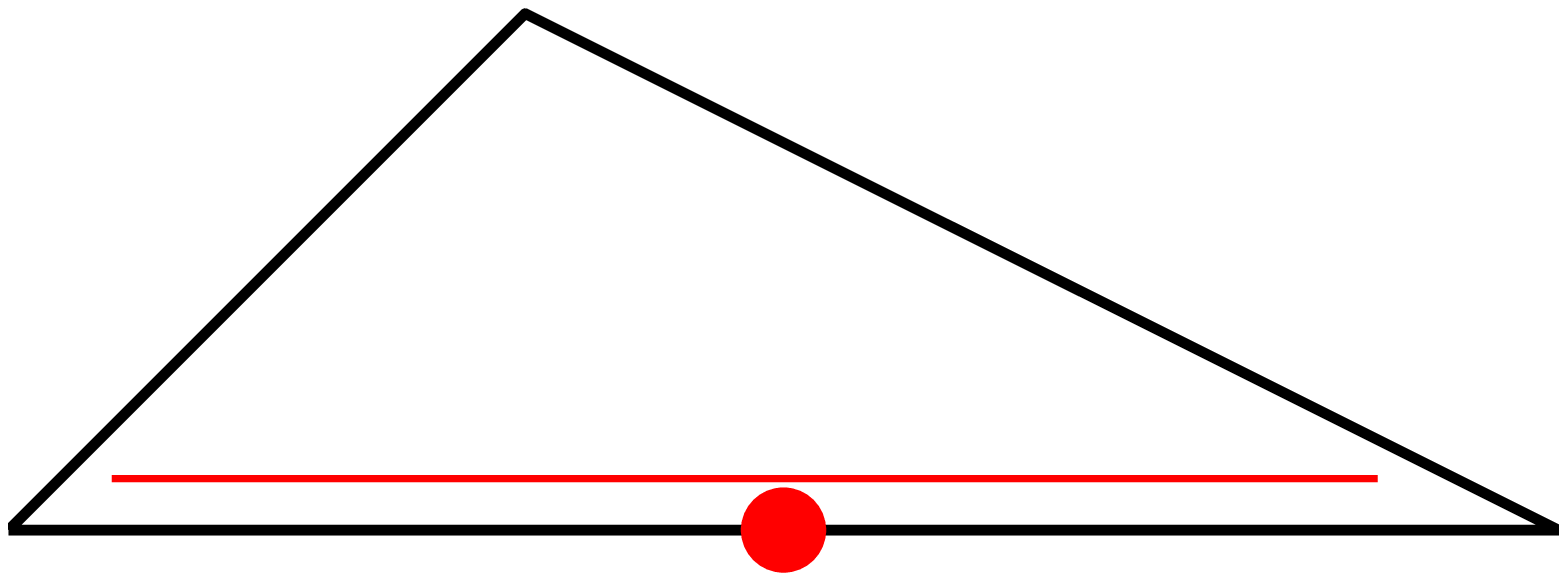} \quad
\includegraphics[width=35mm]{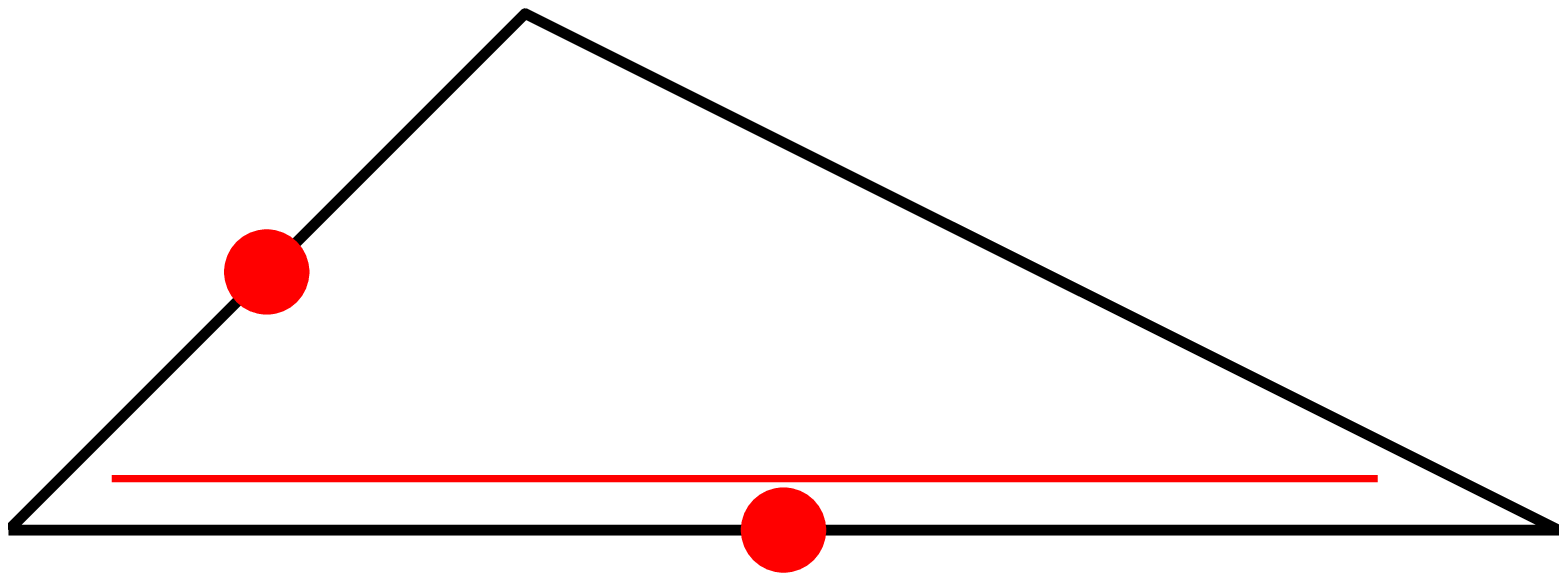} \quad
\includegraphics[width=35mm]{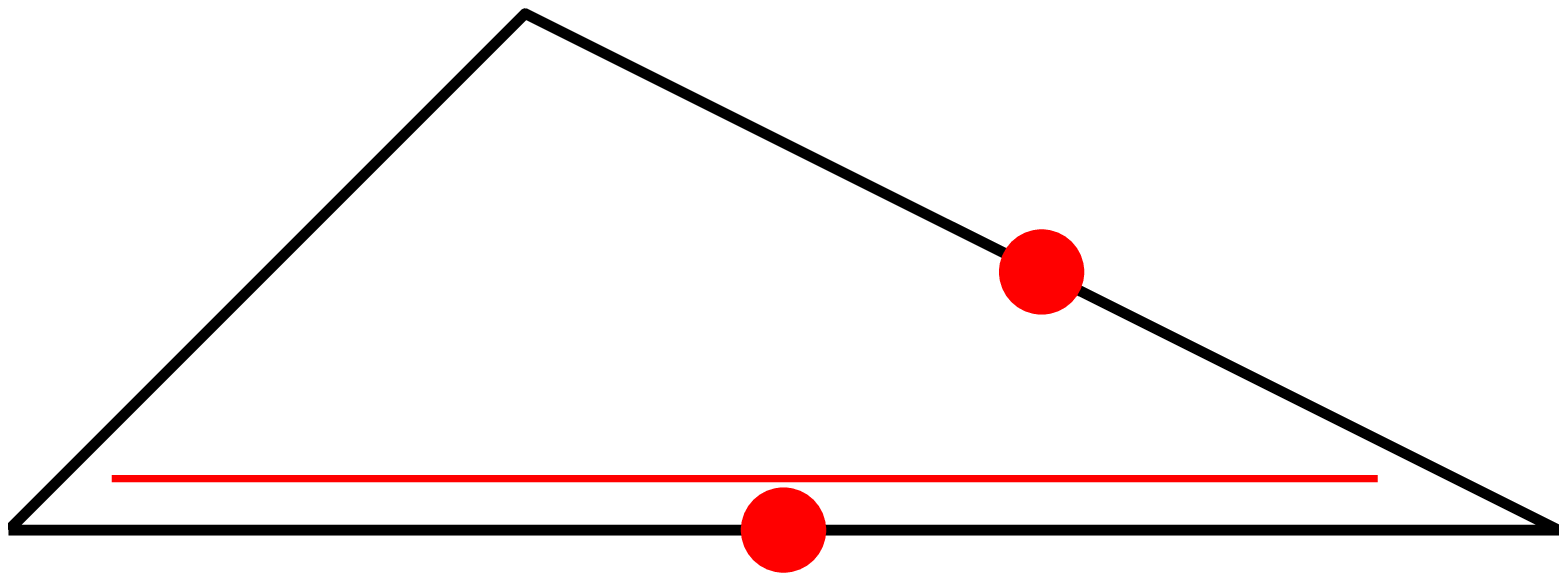} \quad
\includegraphics[width=35mm]{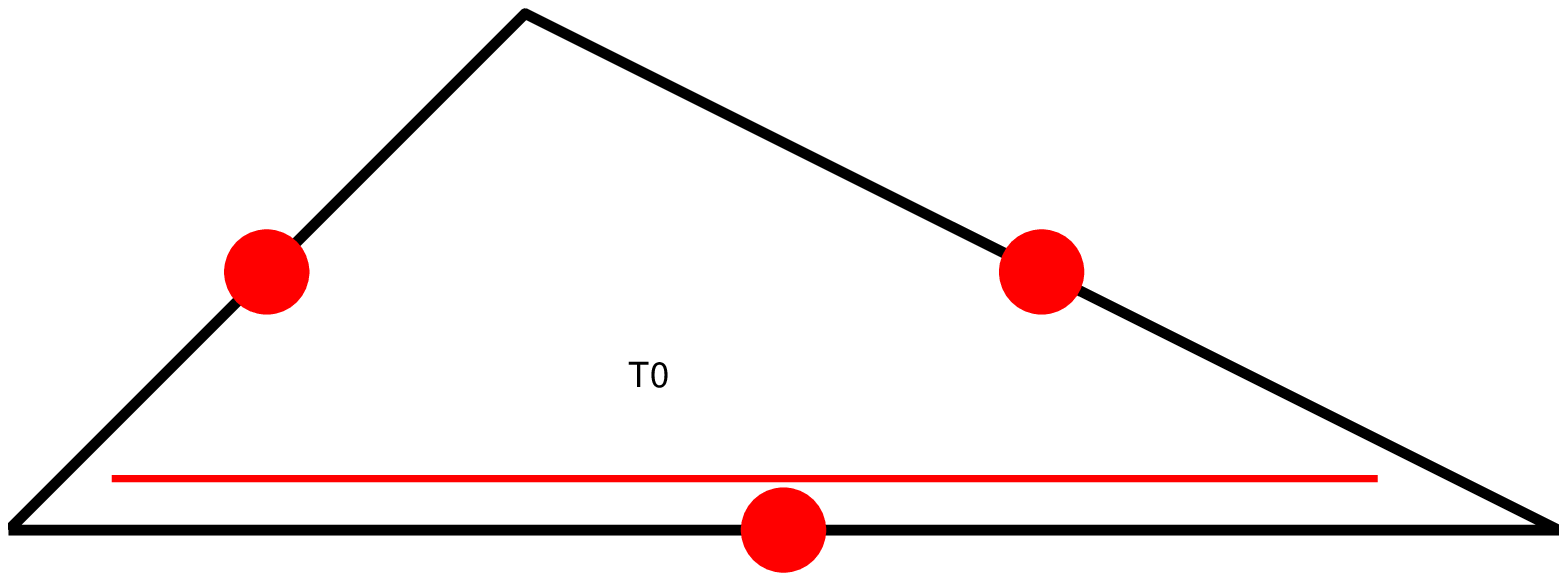} \\
\includegraphics[width=35mm]{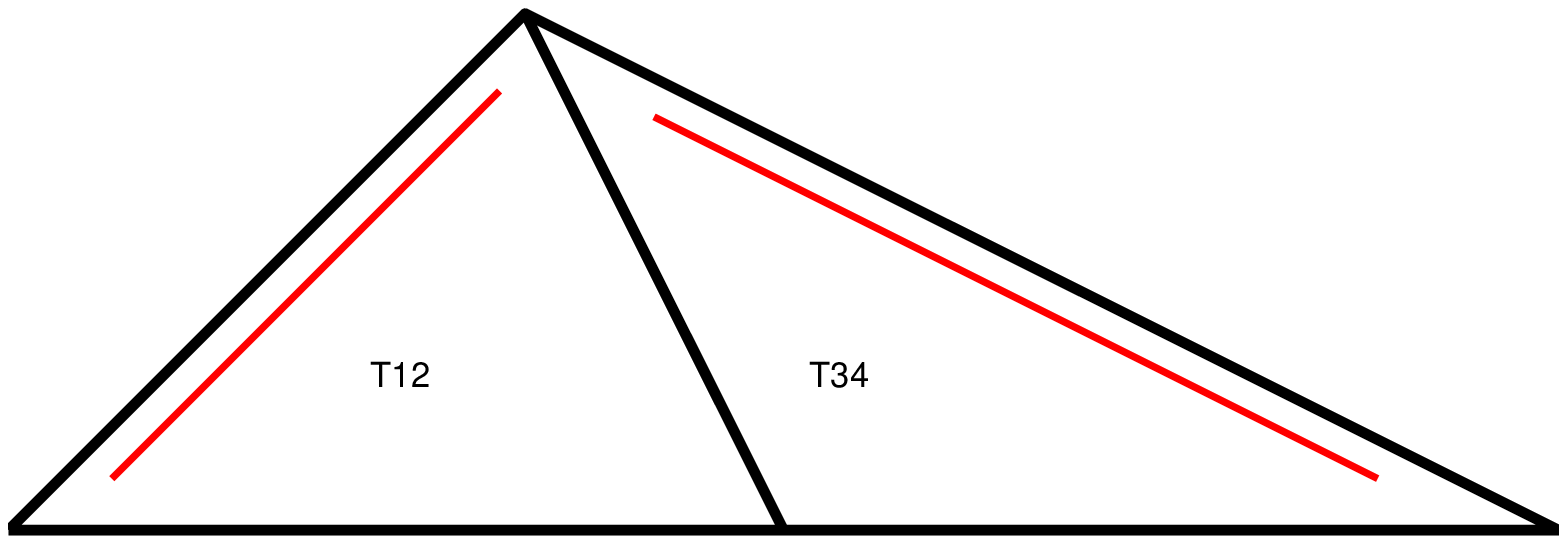} \quad
\includegraphics[width=35mm]{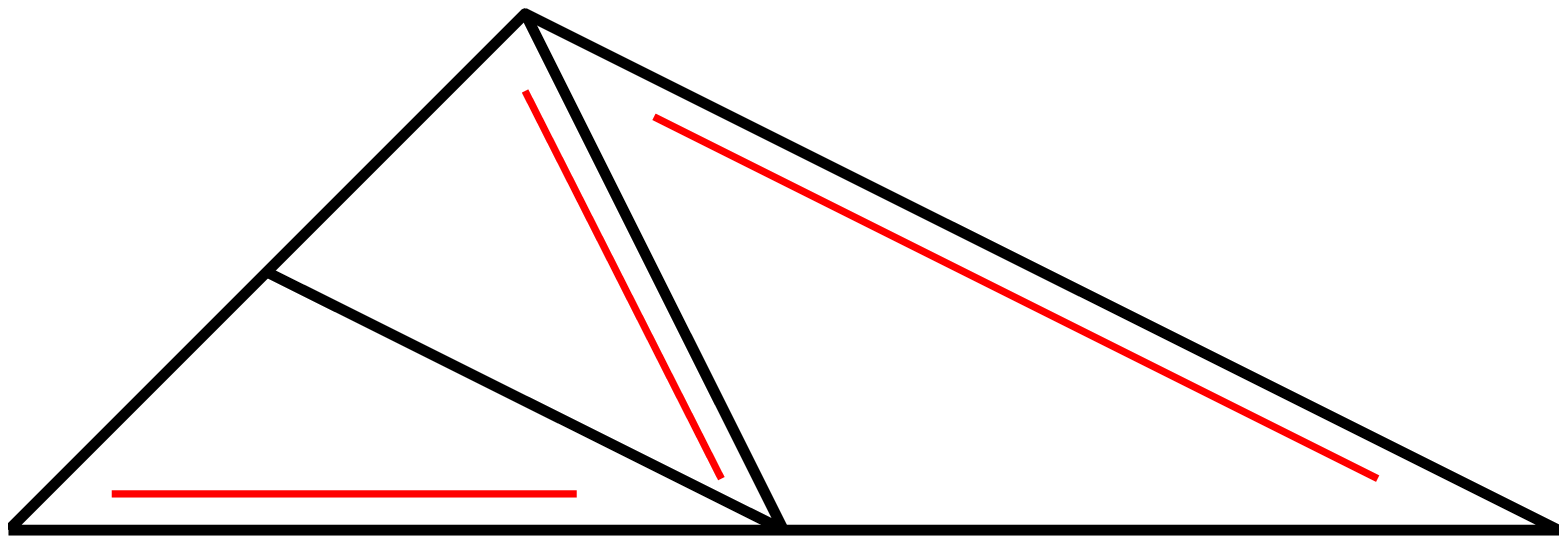}\quad
\includegraphics[width=35mm]{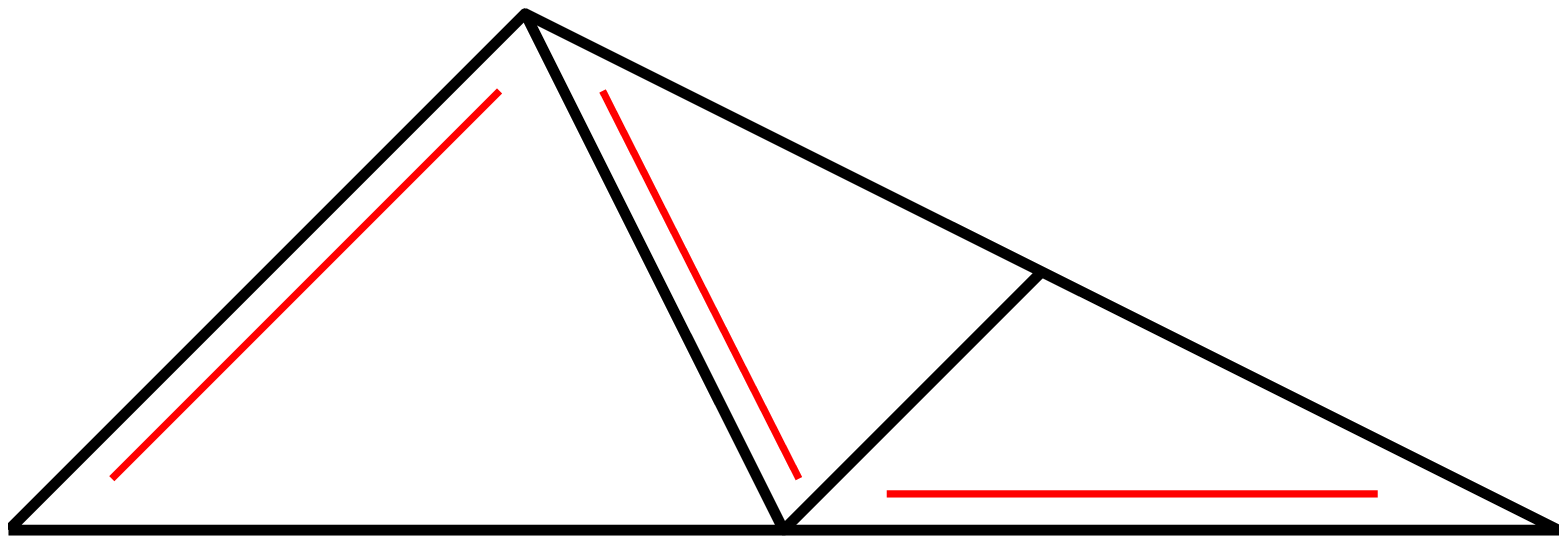}\quad
\includegraphics[width=35mm]{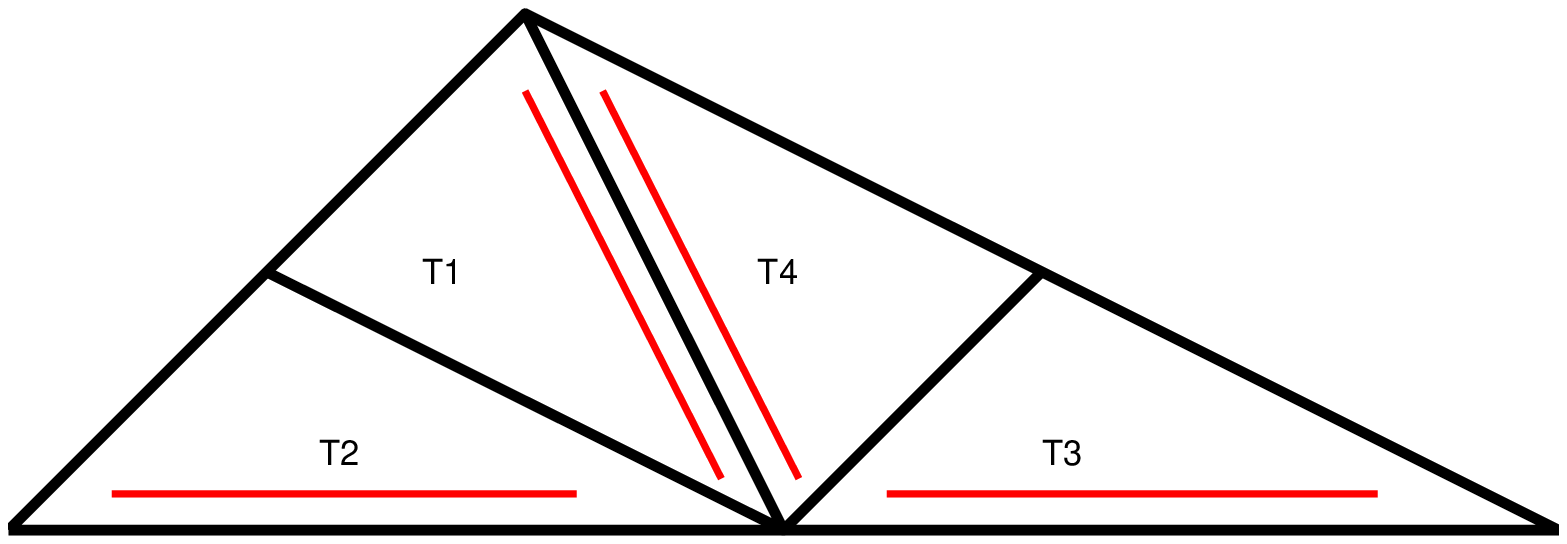}
\caption{
For each triangle $T\in\TT_\ell$, there is one fixed \emph{reference edge},
indicated by the double line (left, top). Refinement of $T$ is done by bisecting
the reference edge, where its midpoint becomes a new node. The reference
edges of the son triangles $T'\in\TT_{\ell+1}$ are opposite to this newest vertex (left, bottom).
To avoid hanging nodes, one proceeds as follows:
We assume that certain edges of $T$, but at least the reference edge,
are marked for refinement (top).
Using iterated newest vertex bisection, the element is then split into
2, 3, or 4 son triangles (bottom).}
\label{fig:nvb}
\end{figure}
%
\subsection{Newest vertex bisection}
\label{section:nvb}%
Throughout, we assume that newest vertex bisection is used for mesh-refinement, see Figure~\ref{fig:nvb}.
Let $\TT_\ell$ be a given mesh and $\MM_\ell\subseteq\EE_\ell$ an arbitrary set of marked edges. Then, \begin{align}
 \TT_{\ell+1}=\refine(\TT_\ell,\MM_\ell)
\end{align}
denotes the coarsest regular triangulation such that all marked edges $E\in\MM_\ell$ have been bisected. Moreover, we write
\begin{align}
 \TT_* = \refine(\TT_\ell)
\end{align}
if $\TT_*$ is a finite refinement of $\TT_\ell$, i.e., there are finitely many triangulations $\TT_{\ell+1},\dots,\TT_n$ and sets of marked edges $\MM_\ell\subseteq\EE_\ell,\dots,\MM_{n-1}\subseteq\EE_{n-1}$ such that $\TT_*=\TT_n$ and $\TT_{j+1}=\refine(\TT_j,\MM_j)$ for all $j=\ell,\dots,n-1$.%

We stress that, for a fixed initial mesh $\TT_0$, only finitely many shapes of triangles $T\in\TT_\ell$ appear. In particular, only finitely many shapes of patches~\eqref{eq:patch:z}--\eqref{eq:patch:E} appear. This observation will be used below. Moreover, newest vertex bisection guarantees that any sequence $\TT_\ell$ of generated meshes with $\TT_{\ell+1}=\refine(\TT_\ell)$ is uniformly shape regular in the sense of
\begin{align}
 \sup_{\ell\in\N}\sigma(\TT_\ell)<\infty,
 \quad\text{where}\quad
 \sigma(\TT_\ell)=\max_{T\in\TT}\frac{\diam(T)^2}{|T|}.
\end{align}%
Further details are found in~\cite[Chapter~4]{verfuerth}.

\subsection{Scott-Zhang quasi-interpolation and discrete lifting operator}
\label{section:sz}%
Our analysis below makes heavy use of the Scott-Zhang projection $P_\ell :H^1(\Omega) \to \SS^1(\TT_\ell)$ from~\cite{sz}: For all nodes $z\in\KK_\ell$, one chooses an edge $E_z\in\EE_\ell$ with $z\in E_z$. For $z\in \Gamma$, this choice is restricted to $E_z \subset \Gamma$. Moreover, for $z\in\overline\Gamma_D$, we even enforce $E_z\subset\overline\Gamma_D$. For $w \in H^1(\Omega)$,
$P_\ell w$ is then defined by
\begin{align*}
 (P_\ell w)(z):= \dual{\psi_z}{w}_{E_z},
\end{align*}
for a node $z\in\KK_\ell$. Here, $\psi_z \in L^2(E_z)$ denotes the dual basis function defined by $\dual{\psi_z}{\varphi_{z^\prime}}_{E_z}=\delta_{zz^\prime}$, and
$\varphi_z\in\SS^1(\TT_\ell)$ denotes the hat function associated with $z\in\KK_\ell$.
By definition, we then have the following projection properties
\begin{itemize}
\item $P_\ell W_\ell = W_\ell$ for all $W_\ell\in\SS^1(\TT_\ell)$,
\item $(P_\ell w)|_\Gamma = w|_\Gamma$ for all $w\in H^1(\Omega)$ and $W_\ell\in\SS^1(\TT_\ell)$ with $w|_\Gamma = W_\ell|_\Gamma$,
\item $(P_\ell w)|_{\Gamma_D} = w|_{\Gamma_D}$ for all $w\in H^1(\Omega)$ and $W_\ell\in\SS^1(\TT_\ell)$ with $w|_{\Gamma_D} = W_\ell|_{\Gamma_D}$,
\end{itemize}
i.e.\ the projection $P_\ell$ preserves discrete (Dirichlet) boundary data.
Moreover, $P_\ell$ satisfies the following stability property
\begin{align}
\label{eq:scottzhangstability}
 \normHe{(1-P_\ell) w}{\Omega} \leq \c{scottzhang}\,\norm{\nabla w}{L^2(\Omega)}\quad \text{for all } w \in H^1(\Omega)
\end{align}
and approximation property
\begin{align}
\label{eq:scottzhangapprox}
 \norm{(1-P_\ell) w}{L^2(\Omega)} \leq \c{scottzhang}\,\norm{h_\ell\nabla w}{L^2(\Omega)}\quad \text{for all } w \in H^1(\Omega)
\end{align}
where $\setc{scottzhang}>0$ depends only on $\sigma(\TT_\ell)$.
Together with the projection property onto $\SS^1(\TT_\ell)$, it is an easy consequence of the stability~\eqref{eq:scottzhangstability} of $P_\ell$ that
\begin{align}
 \label{eq:quasioptscottzhang}
 \normHe{(1-P_\ell)w}{\Omega}
 =\min_{W_\ell\in\SS^1(\TT_\ell)}\normHe{(1-P_\ell)(w-W_\ell)}{\Omega}
 \lesssim\min_{W_\ell\in\SS^1(\TT_\ell)}\norm{\nabla(w-W_\ell)}{L^2(\Omega)}
\end{align}
for all $w\in H^1(\Omega)$. In particular, $P_\ell$ is quasi-optimal in the sense of the C\'ea lemma with respect to $\norm{\cdot}{H^1(\Omega)}$ and $\norm{\nabla(\cdot)}{L^2(\Omega)}$, i.e.
\begin{align}
\begin{split}
 \normHe{(1-P_\ell)w}{\Omega}
 &\lesssim\min_{W_\ell\in\SS^1(\TT_\ell)}\norm{w-W_\ell}{H^1(\Omega)},\\
 \norm{\nabla(1-P_\ell)w}{L^2(\Omega)}
 &\lesssim\min_{W_\ell\in\SS^1(\TT_\ell)}\norm{\nabla(w-W_\ell)}{L^2(\Omega)}.
\end{split}
\end{align}
Moreover, $P_\ell$ allows to define a discrete lifting operator
\begin{align}
\label{eq:discreteLifting}
 \LL_\ell := P_\ell \LL \,:\,\SS^1(\EE_\ell^\Gamma) \to \SS^1(\TT_\ell),
 \quad\text{i.e. }
 \LL_\ell(W_\ell|_\Gamma)|_\Gamma = W_\ell|_\Gamma
 \quad\text{for all }W_\ell\in\SS^1(\TT_\ell)
\end{align}
whose operator norm is uniformly bounded in terms of $\sigma(\TT_\ell)$. Here, $\LL\in L(H^{1/2}(\Gamma); H^1(\Omega))$ denotes an arbitrary lifting operator, i.e.\ $(\LL w)|_\Gamma = w$ for all $w\in H^{1/2}(\Gamma)$, see e.g.\ \cite{mclean}.

Finally, we put emphasis on the fact that our definition of $P_\ell$ also provides an operator $P_\ell = P_\ell^\Gamma :L^2(\Gamma) \to \SS^1(\EE_\ell^\Gamma)$ which is consistent in the sense that $(P_\ell v)|_\Gamma = P_\ell^\Gamma (v|_\Gamma)$ for all $v\in H^1(\Omega)$.
Using the definition of $H^{1/2}(\Gamma)$ as the trace space of $H^1(\Omega)$
and the stability~\eqref{eq:scottzhangstability}, we see
\begin{align*}
 \norm{\widehat g-P_\ell \widehat g}{H^{1/2}(\Gamma)}
 &:= \inf\set{\normHe{w}{\Omega}}{w\in H^1(\Omega), w|_{\Gamma} = \widehat g-P_\ell \widehat g}\\
 &\le \inf\set{\norm{w-P_\ell w}{H^1(\Omega)}}{w\in H^1(\Omega), w|_\Gamma=\widehat g}\\
 &\lesssim\inf\set{\norm{\nabla w}{L^2(\Omega)}}{w\in H^1(\Omega), w|_\Gamma=\widehat g}\\
 &\le \inf\set{\norm{w}{H^1(\Omega)}}{w\in H^1(\Omega), w|_\Gamma=\widehat g}
 = \norm{\widehat g}{H^{1/2}(\Gamma)}
\end{align*}
for all $\widehat g\in H^{1/2}(\Gamma)$, i.e.\ $P_\ell:H^{1/2}(\Gamma)\to\SS^1(\EE_\ell^\Gamma)$ is a continuous projection with respect to the $H^{1/2}$-norm. In particular, $P_\ell$ also provides a continuous projection $P_\ell = P_\ell^D:H^{1/2}(\Gamma_D)\to\SS^1(\EE_\ell^{D})$, since
\begin{align*}
 \norm{g-P_\ell g}{H^{1/2}(\Gamma_D)}
 &= \inf\set{\norm{\widehat g-P_\ell \widehat g}{H^{1/2}(\Gamma)}}{\widehat g\in H^{1/2}(\Gamma), \widehat g|_{\Gamma_D}=g}\\
 &\lesssim \inf\set{\norm{\widehat g}{H^{1/2}(\Gamma)}}{\widehat g\in H^{1/2}(\Gamma), \widehat g|_{\Gamma_D}=g}
 = \norm{g}{H^{1/2}(\Gamma_D)}
\end{align*}
for all $g\in H^{1/2}(\Gamma_D)$. As before, this definition is consistent with the previous notation of $P_\ell$ since $(P_\ell^\Gamma\widehat g)|_{\Gamma_D} = P_\ell^D(\widehat g|_{\Gamma_D})$ for all $\widehat g\in H^{1/2}(\Gamma)$.

\section{A~Posteriori Error Estimation and Adaptive Mesh-Refinement}
\label{section:aposteriori}%

\vspace*{-3mm}
\subsection{Data oscillations}
\label{section:oscillations}%
We start with the element data oscillations
\begin{align}
 \oscT{\ell}^2:=\sum_{T\in\TT_\ell}\oscT{\ell}(T)^2,
 \text{ where }
 \oscT{\ell}(T)^2:=|T|\,\normL2{f-f_T}{T}^2 \quad \text{for all } T\in\TT_\ell
\end{align}
and where $f_T:=|T|^{-1}\int_T f\,dx\in\R$ denotes the integral mean over an element $T\in\TT_\ell$. These arise in the efficiency estimate for residual error estimators.

Our residual error estimator will involve the edge data oscillations
\begin{align}\label{eq:osc:edge}
 \osc{\ell}^2 := \sum_{E\in\EE_\ell^\Omega}\osc{\ell}(E)^2,
 \text{ where }
 \osc{\ell}(E)^2 := |\omega_{\ell,E}|\,\norm{f-f_{\omega_{\ell,E}}}{L^2(\omega_{\ell,E})}^2
 \text{ for all }E\in\EE_\ell^\Omega.
\end{align}
Here, $\omega_{\ell,E}\subset\Omega$ is the edge patch from~\eqref{eq:patch:E}, and $f_{\omega_{\ell,E}}\in\R$ is the corresponding integral mean of $f$.

For the analysis, we shall additionally need the node data oscillations
\begin{align}\label{eq:osc:node}
 \oscK{\ell}^2 := \sum_{z\in\KK_\ell^\Omega}\oscK{\ell}(z)^2,
 \text{ where }
 \oscK{\ell}(z)^2 := |\omega_{\ell,z}|\,\norm{f-f_{\omega_{\ell,z}}}{L^2(\omega_{\ell,z})}^2
 \text{ for all }z\in\KK_\ell^\Omega.
\end{align}
Here, $\omega_{\ell,z}\subset\Omega$ is the node patch from~\eqref{eq:patch:z}, and $f_{\omega_{\ell,z}}\in\R$ is the corresponding integral mean of $f$.

Moreover, the efficiency needs the Neumann data oscillations
\begin{align}\label{eq:osc:neu}
 \oscN{\ell}^2:=\sum_{E\in\EE_\ell^N}\oscN\ell(E)^2,
 \text{ where }
 \oscN{\ell}(E)^2:=|E|\,\normL2{\phi-\phi_E}{E}^2
 \text{ for all } E \in \EE_\ell^N
\end{align}
and where $\phi_E:=|E|^{-1}\int_E \phi \,dx$ denotes the integral mean over an edge $E\in\EE_\ell^N$.

Finally, the approximation of the Dirichlet data $g\approx g_\ell$ is controlled by the Dirichlet data oscillations
\begin{align}\label{eq:osc:dir}
 \oscD{\ell}:=\sum_{E\in\EE_\ell^D} \oscD{\ell}(E)^2,
 \text{ where }
 \oscD{\ell}(E)^2:=|E|\normL2{(g-g_\ell)^\prime}{E}^2
 \text{ for all }E\in\EE_\ell^D.
\end{align}
\new{Recall that, on the 1D manifold $\Gamma_D$, the derivative of the nodal interpoland is the elementwise best approximation of the derivative by piecewise constants, i.e.,
\begin{align}\label{eq:bestapprox}
 \norm{(g-g_\ell)'}{L^2(E)} = \min_{c\in\R}\norm{g'-c}{L^2(E)}
 \quad\text{for all }E\in\EE_\ell^D.
\end{align}
According to the elementwise Pythagoras theorem, this implies
\begin{align}\label{eq:nodal:orthogonality}
 \norm{(g-g_\ell)'}{L^2(E)}^2 + \norm{(g_\ell-\widetilde g_\ell)'}{L^2(E)}^2
 = \norm{(g-\widetilde g_\ell)'}{L^2(E)}^2
 \text{ for all }\widetilde g_\ell\in\SS^1(\EE_\ell^D)
\end{align}
and all Dirichlet edges $E\in\EE_\ell^D$. This observation will be crucial in the analysis below.
Moreover,~\eqref{eq:bestapprox} yields  
\begin{align}
\norm{h_\ell^{1/2}(g-g_\ell)^\prime}{L^2(\Gamma_D)} = \min_{W_\ell \in \SS^1(\TT_\ell)}\norm{h_\ell^{1/2}(g-W_\ell|_\Gamma)^\prime}{L^2(\Gamma_D)} .
\end{align}
The following result is found in~\cite[Lemma~2.2]{efgp}.
\begin{lemma}\label{lemma:apx}
Let $g\in H^1(\Gamma_D)$ and let $g_\ell$ denote the nodal interpoland of $g_\ell$ on $\overline\Gamma_D$. Then,
\begin{align}\label{eq:apx}
 \norm{g-g_\ell}{H^{1/2}(\Gamma_D)}
 \le \c{apx}\,\oscD\ell,
\end{align}
where the constant $\setc{apx}>0$ depends only on the shape regularity constant $\sigma(\TT_\ell)$ and $\Omega$.\qed
\end{lemma}
}

To keep the notation simple, we extend the Dirichlet and the Neumann data oscillations from~\eqref{eq:osc:neu}--\eqref{eq:osc:dir} by zero to all edges $E\in\EE_\ell$, e.g.\ $\oscD{\ell}(E) = 0$ for $E\in\EE_\ell\backslash\EE_\ell^D$.
Moreover, we will write
\begin{align}
\label{eq:defonpatch}
 \oscT\ell(\omega_{\ell,z})^2 = \sum_{{T\in\TT_\ell}\atop{T\subset\omega_{\ell,z}}}\oscT{\ell}(T)^2
 \quad\text{resp.}\quad
 \oscN{\ell}(\EE_{\ell,z})^2 = \sum_{{E\in\EE_\ell^N}\atop{E\subset\EE_{\ell,z}}}\oscN{\ell}(E)^2
\end{align}
to abbreviate the notation.

\subsection{Element-based residual error estimator}
\label{section:estimator:T}%
Our first proposition states reliability and efficiency of the error estimator $\rho_\ell$ from~\eqref{eq1:estimator:T}--\eqref{eq2:estimator:T}.

\begin{proposition}[reliability and efficiency of $\rho_\ell$]
\label{prop:reliability:rho}
The error estimator $\rho_\ell$ is reliable
\begin{align}\label{eq:rho:reliable}
 \norm{u - U_\ell}{H^1(\Omega)}
 \le \c{rho:reliable}\,\rho_\ell
\end{align}
and efficient
\begin{align}\label{eq:rho:efficient}
 \c{rho:efficient}^{-1}\,\rho_\ell
 \le \big(\norm{\nabla(u-U_\ell)}{L^2(\Omega)}^2 + \oscT{\ell}^2 +\oscN{\ell}^2 + \oscD{\ell}^2\big)^{1/2}.
\end{align}
The constants $\setc{rho:reliable},\setc{rho:efficient}>0$ depend
only on the shape regularity constant $\sigma(\TT_\ell)$ and on $\Omega$.
\end{proposition}

\begin{proof}[Sketch of proof]
We consider a continuous auxiliary problem
\begin{align}
\begin{split}
 -\Delta w &= 0 \hspace*{12.8mm} \text{in } \Omega,\\
 w &= g-g_\ell \quad \text{on } \Gamma_D,\\
 \partial_n w &= 0 \hspace*{12.8mm} \text{on } \Gamma_N,
\end{split}
\end{align}
with unique solution $w\in H^1(\Omega)$. We then have norm equivalence $\norm{w}{H^1(\Omega)}\simeq\norm{g-g_\ell}{H^{1/2}(\Gamma_D)}$ as well as $u-U_\ell-w\in H^1_D(\Omega)$. From this, we obtain
\begin{align*}
 \norm{u-U_\ell}{H^1(\Omega)}^2
 \lesssim \norm{\nabla(u-U_\ell-w)}{L^2(\Omega)}^2 + \norm{g-g_\ell}{H^{1/2}(\Gamma_D)}^2.
\end{align*}
Whereas the second term is controlled by Lemma~\ref{lemma:apx}, the first can be handled as for homogeneous Dirichlet data, i.e. use of the Galerkin orthogonality combined with approximation estimates for a Cl\'ement-type quasi-interpolation operator. Details are found e.g.\ in~\cite{bcd}. This proves reliability~\eqref{eq:rho:reliable}.

By use of bubble functions and local scaling arguments, one obtains the estimates
\begin{align*}
 |T|\,\norm{f}{L^2(T)}^2
 &\lesssim \norm{\nabla(u-U_\ell)}{L^2(T)}^2 + \oscT{\ell}(T)^2+\oscN{\ell}(\partial T\cap \Gamma_N),\\
 |T|^{1/2}\,\norm{[\partial_nU_\ell]}{L^2(E\cap\Omega)}^2
 &\lesssim \norm{\nabla(u-U_\ell)}{L^2(\omega_{\ell,E})}^2 + \oscT{\ell}(\omega_{\ell,E})^2\\
 |T|^{1/2}\,\norm{\phi-\partial_nU_\ell}{L^2(E\cap\Gamma_N)}^2
 &\lesssim \norm{\nabla(u-U_\ell)}{L^2(\omega_{\ell,E})}^2 + \oscT{\ell}(\omega_{\ell,E})^2+\oscN{\ell}(E\cap\Gamma_N)^2 \,
\end{align*}
where $\omega_{\ell,E}$ denotes the edge patch of $E\in\EE_\ell$.
Details are found e.g.\ in~\cite{ao,verfuerth}.
Summing these estimates over all elements, one obtains the efficiency estimate~\eqref{eq:rho:efficient}.
\end{proof}

\begin{proposition}[discrete local reliability of $\rho_\ell$]
\label{prop:dlr:rho}
Let $\TT_* = \refine(\TT_\ell)$ be an arbitrary refinement of $\TT_\ell$ with associated Galerkin solution $U_*\in\SS^1(\TT_*)$. Let $\RR_\ell(\TT_*):=\TT_\ell\backslash\TT_*$ be
the set of all elements $T\in\TT_\ell$ which are refined to generate $\TT_*$. Then, there holds
\begin{align}\label{eq:dlr:rho}
 \norm{U_*-U_\ell}{H^1(\Omega)}
 \le \c{rho:dlr}\,\rho_\ell(\RR_\ell(\TT_*))
\end{align}
with some constant $\setc{rho:dlr}>0$ which depends only on $\sigma(\TT_\ell)$ and $\Omega$.
\end{proposition}

\begin{proof}
We consider a discrete auxiliary problem
\begin{align*}
 \dual{\nabla W_*}{\nabla V_*}_\Omega = 0
 \quad\text{for all }V_*\in\SS^1_D(\TT_*)
\end{align*}
with unique solution $W_*\in\SS^1(\TT_*)$ with $W_*|_{\Gamma_D}=g_*-g_\ell$.
To estimate the $H^1$-norm of $W_*$ in terms of the boundary data,
let $\LL_*\,:H^{1/2}(\Gamma) \to \SS^1(\TT_*)$ denote the discrete lifting operator from~\eqref{eq:discreteLifting}. Let $\widehat g_*,\widehat g_\ell\in H^{1/2}(\Gamma)$
be arbitrary extensions of $g_*$ and $g_\ell$, respectively.
Then, we have $V_*=W_* - \LL_*(\widehat g_*-\widehat g_\ell)\in\SS^1_D(\TT_*)$.
According to the triangle inequality and a Poincar\'e inequality for $V_*\in\SS^1_D(\TT_*)$, we first observe
\begin{align*}
 \norm{W_*}{L^2(\Omega)}
 &\le \norm{V_*}{L^2(\Omega)}
 + \norm{\LL_*(\widehat g_*-\widehat g_\ell)}{L^2(\Omega)}\\
 &\lesssim \norm{\nabla V_*}{L^2(\Omega)}
 + \norm{\LL_*(\widehat g_*-\widehat g_\ell)}{L^2(\Omega)}\\
 & \lesssim \norm{\nabla W_*}{L^2(\Omega)} + \norm{\LL_*(\widehat g_*-\widehat g_\ell)}{H^1(\Omega)}.
\end{align*}
Moreover, the variational formulation for $W_*\in\SS^1(\TT_*)$ yields
\begin{align*}
 0
 = \dual{\nabla W_*}{\nabla V_*}_\Omega
 = \norm{\nabla W_*}{L^2(\Omega)}^2
 - \dual{\nabla W_*}{\nabla\LL_*(\widehat g_*-\widehat g_\ell)}_\Omega,
\end{align*}
whence by the Cauchy-Schwarz inequality
\begin{align*}
 \norm{\nabla W_*}{L^2(\Omega)}
 \le \norm{\nabla\LL_*(\widehat g_*-\widehat g_\ell)}{L^2(\Omega)}
 \lesssim \norm{\widehat g_*-\widehat g_\ell}{H^{1/2}(\Gamma)}.
\end{align*}
Altogether, this proves $\norm{W_*}{H^1(\Omega)}\lesssim \norm{\widehat g _*-\widehat g_\ell}{H^{1/2}(\Gamma)}$. Since the extensions $\widehat g_*,\widehat g_\ell$ were arbitrary and by definition of the $H^{1/2}(\Gamma_D)$-norm, this proves
\begin{align}\label{dpr1:rho}
 \norm{W_*}{H^1(\Omega)}
 \lesssim \norm{g_*-g_\ell}{H^{1/2}(\Gamma_D)}
 \lesssim \norm{h_\ell^{1/2}(g_*-g_\ell)'}{L^2(\Gamma_D)},
\end{align}
where we have finally used that $g_\ell$ is also the nodal interpoland of $g_*$ so that Lemma~\ref{lemma:apx} applies.
For an element $T\in\TT_\ell\cap\TT_*$ holds $g_*|_{\partial T\cap\Gamma_D} = g_\ell|_{\partial T\cap\Gamma_D}$, and the last term thus satisfies
\begin{align*}
 \norm{h_\ell^{1/2}(g_*-g_\ell)'}{L^2(\Gamma_D)}^2
 \simeq
 \sum_{T\in\TT_\ell} |T|^{1/2}\norm{(g_*-g_\ell)'}{L^2(\partial T\cap\Gamma_D)}^2
 &=\!\! \sum_{T\in\RR_\ell(\TT_*)} \!\! |T|^{1/2}\norm{(g_*-g_\ell)'}{L^2(\partial T\cap\Gamma_D)}^2.
\end{align*}
With the orthogonality relation~\eqref{eq:nodal:orthogonality} applied for $g_*\in\SS^1(\TT_*|_{\Gamma_D})$, we see
\begin{align*}
 \norm{W_*}{H^1(\Omega)}^2
 \lesssim \sum_{T\in\RR_\ell(\TT_*)} \!\!|T|^{1/2}\norm{(g_*-g_\ell)'}{L^2(\partial T\cap\Gamma_D)}^2
 \le \sum_{T\in\RR_\ell(\TT_*)} \!\! |T|^{1/2}\norm{(g-g_\ell)'}{L^2(\partial T\cap\Gamma_D)}^2.
\end{align*}
Finally, we observe $U_*-U_\ell-W_*\in\SS^1_D(\TT_*)$ with
\begin{align*}
 \dual{\nabla(U_*-U_\ell-W_*)}{\nabla V_\ell} = 0
 \quad\text{for all }V_\ell\in\SS^1_D(\TT_\ell).
\end{align*}
Arguing as in~\cite[Lemma~3.6]{ckns}, we see
\begin{align*}
 &\norm{\nabla(U_*-U_\ell-W_*)}{L^2(\Omega)}^2
 \\&\quad
 \lesssim \sum_{T\in\RR_\ell(\TT_*)}\big(|T|\,\norm{f}{L^2(T)}^2
 + |T|^{1/2}\,\norm{[\partial_nU_\ell]}{L^2(\partial T\cap\Omega)}^2
 + |T|^{1/2}\,\norm{\phi-\partial_nU_\ell}{L^2(\partial T\cap\Gamma_N)}^2\big)
\end{align*}
Finally, we again use the triangle inequality and the Poincar\'e inequality to see
\begin{align*}
 \norm{U_*-U_\ell}{H^1(\Omega)}^2
 \lesssim \norm{W_*}{H^1(\Omega)}^2 + \norm{\nabla(U_*-U_\ell-W_*)}{L^2(\Omega)}^2
\end{align*}
and thus obtain the discrete local reliability~\eqref{eq:dlr:rho}.
The constant $\c{rho:dlr}>0$ depends only on $\c{apx}>0$ and on local estimates for the Scott-Zhang projection which are controlled by boundedness of $\sigma(\TT_\ell)$.
\end{proof}
\subsection{Edge-based residual error estimator}
\label{section:estimator:E}%
In the following, 
we show that the edge-based estimator $\varrho_\ell$ from~\eqref{eq1:estimator:E}--\eqref{eq2:estimator:E} is locally equivalent
to the element-based error estimator $\rho_\ell$ from the previous section. The main advantage is that $\varrho_\ell$ replaces the volume residuals
\begin{align}
 \res_\ell(T):=|T|\,\norm{f}{L^2(T)}
\end{align}
by the edge oscillations $\osc\ell$.
We define the edge jump contributions
\begin{align}
 \eta_\ell(E)^2 := \begin{cases}
  |E|\,\norm{[\partial_nU_\ell]}{L^2(E)}^2
  \quad&\text{for }E\in\EE_\ell^\Omega,\\
  |E|\,\norm{\phi-\partial_nU_\ell}{L^2(E)}^2
  \quad&\text{for }E\in\EE_\ell^N
 \end{cases}
\end{align}
where $[\cdot]$ denotes the jump across an interior edge.
Together with the edge oscillations from~\eqref{eq:osc:edge} and the Dirichlet oscillations from~\eqref{eq:osc:dir}, our version of the residual error estimator from~\eqref{eq1:estimator:E}--\eqref{eq2:estimator:E} reads
\begin{align}
 \varrho_\ell^2
 = \sum_{E\in\EE_\ell}\varrho_\ell(E)^2 = \sum_{E\in\EE_\ell^\Omega\cup\EE_\ell^N}\eta_\ell(E)^2
 + \sum_{E\in\EE_\ell^\Omega}\osc{\ell}(E)^2 + \sum_{E\in\EE_\ell^D}\oscD{\ell}(E)^2.
\end{align}%
Note that $\osc{\ell}(\EE_{\ell,z})$, $\eta_\ell(\EE_{\ell,z})$, and $\res_\ell(\omega_{\ell,E})$ are defined analogously to~\eqref{eq:defonpatch}.
The following lemma implies local equivalence of the estimators $\rho_\ell$ and $\varrho_\ell$.

\begin{lemma}\label{lemma:local}
The following local estimates hold:
\begin{itemize}
\item[{\rm(i)}] $\oscT{\ell}(\omega_{\ell,E}) \le \osc{\ell}(E) \le \c{eq1} \res_\ell(\omega_{\ell,E})$ for all $E\in\EE_\ell^\Omega$.
\item[{\rm(ii)}]
$\res_\ell(\omega_{\ell,z}) \le \c{eq2} \big(\eta_\ell(\EE_{\ell,z})
+ \oscK{\ell}(z)\big)$ for all $z\in\KK_\ell^\Omega$.
\item[{\rm(iii)}]
$\c{eq3}^{-1}\,\osc{\ell}(\EE_{\ell,z}) \le \oscK{\ell}(z) \le \c{eq4}\,\osc{\ell}(\EE_{\ell,z})$ for all $z\in\KK_\ell^\Omega$.%
\end{itemize}
The constants $\setc{eq1},\setc{eq2},\setc{eq3}>0$ depend only on the shape regularity constant $\sigma(\TT_\ell)$, whereas $\setc{eq4}>0$ depends on the use of newest vertex bisection and the initial mesh $\TT_0$.
\end{lemma}

\begin{proof}[Sketch of proof]
The proof of (i) follows from the fact that taking the integral mean $f_\omega$ is the $L^2$ best approximation by a constant, i.e.
\begin{align*}
 \norm{f-f_\omega}{L^2(\omega)} = \min_{c\in\R}\norm{f-c}{L^2(\omega)}
 \quad\text{for all measurable }\omega\subseteq\Omega,
\end{align*}
and that the area of neighboring elements can only change up to $\sigma(\TT_\ell)$. The estimate (ii) is well-known and found, e.g., in~\cite[Section 2.2.4]{ks}. Note that
(ii) essentially needs the condition that each element $T\in\TT_\ell$ has an interior node, cf. Section~\ref{sec:notation}. The lower estimate in (iii) follows from the same arguments as (i), namely
\begin{align*}
 \norm{f-f_{\omega_{\ell,E}}}{L^2(\omega_{\ell,E})}
 \le \norm{f-f_{\omega_{\ell,z}}}{L^2(\omega_{\ell,E})}
 \le \norm{f-f_{\omega_{\ell,z}}}{L^2(\omega_{\ell,z})}
\end{align*}
and the fact that ---up to shape regularity--- only finitely many edges belong to $\EE_{\ell,z}$.
For $f$ being a piecewise polynomial, the upper estimate in (iii) follows from a scaling argument since both terms,
$\osc{\ell}(\EE_{\ell,z})\simeq \oscK{\ell}(z)$ define seminorms on $\PP^p(\set{T\in\TT_\ell}{z\in T})$ with kernel being the constant functions. Note that the equivalence constants depend on the shape of the node patch $\omega_{\ell,z}$, but newest vertex bisection leads only to finitely many shapes of the patches. For arbitrary $f\in L^2(\Omega)$, we first observe that the $\TT_\ell$-piecewise integral mean $f_\ell\in\PP^0(\TT_\ell)$, defined by $f_\ell|_T = f_T$ for all $T\in\TT_\ell$,
satisfies $(f_\ell)_{\omega_{\ell,E}} = f_{\omega_{\ell,E}}$ as well as
$(f_\ell)_{\omega_{\ell,z}} = f_{\omega_{\ell,z}}$, e.g.
\begin{align*}
 (f_\ell)_{\omega_{\ell,z}}
 = \frac{1}{|\omega_{\ell,z}|}\int_{\omega_{\ell,z}} f_\ell\,dx
 = \frac{1}{|\omega_{\ell,z}|}\sum_{T\subset\omega_{\ell,z}}\int_T f_\ell\,dx
 = \frac{1}{|\omega_{\ell,z}|}\sum_{T\subset\omega_{\ell,z}}\int_T f\,dx
 = f_{\omega_{\ell,z}}.
\end{align*}
This and the Pythagoras theorem for the integral mean $f_\ell$ prove
\begin{align*}
 \norm{f-f_{\omega_{\ell,z}}}{L^2(\omega_{\ell,z})}^2
 &= \norm{f-f_\ell}{L^2(\omega_{\ell,z})}^2
 + \norm{f_\ell-f_{\omega_{\ell,z}}}{L^2(\omega_{\ell,z})}^2\\
 &\lesssim \sum_{E\in\EE_{\ell,z}}\norm{f-f_\ell}{L^2(\omega_{\ell,E})}^2
 + \sum_{E\in\EE_{\ell,z}}\norm{f_\ell-f_{\omega_{\ell,z}}}{L^2(\omega_{\ell,E})}^2\\
 &= \sum_{E\in\EE_{\ell,z}}\norm{f-f_{\omega_{\ell,z}}}{L^2(\omega_{\ell,E})}^2.
\end{align*}
Scaling with $|\omega_{\ell,z}|\simeq|\omega_{\ell,E}|$ concludes the proof.
\end{proof}

\begin{proposition}[reliability and efficiency of $\varrho_\ell$]\label{prop:reliability}
The error estimator $\varrho_\ell$ is reliable
\begin{align}\label{eq:reliable}
 \norm{u - U_\ell}{H^1(\Omega)}
 \le \c{reliable}\,\varrho_\ell
\end{align}
and efficient
\begin{align}\label{eq:efficient}
 \c{efficient}^{-1}\,\varrho_\ell \le \big(\norm{\nabla(u-U_\ell)}{L^2(\Omega)}^2 + \oscT{\ell}^2 +\oscN{\ell}^2 + \oscD{\ell}^2\big)^{1/2}.
\end{align}
The constants $\setc{reliable},\setc{efficient}>0$ depend
only on $\Omega$, the use of newest
vertex bisection, and the initial mesh $\TT_0$.
\end{proposition}

\begin{proof}
With the help of the preceding lemma, we obtain equivalence
$\varrho_\ell \simeq \rho_\ell$. Consequently, reliability and efficiency of $\varrho_\ell$ follow from the respective properties of the element-based estimator $\rho_\ell$,
see Proposition~\ref{prop:reliability:rho}.
\end{proof}

\begin{proposition}[discrete local reliability of $\varrho_\ell$]
\label{prop:dlr}
Let $\TT_* = \refine(\TT_\ell)$ be an arbitrary refinement of $\TT_\ell$ with associated Galerkin solution $U_*\in\SS^1(\TT_*)$. Let $\RR_\ell(\TT_*):=\TT_\ell\backslash\TT_*$ be
the set of all elements $T\in\TT_\ell$ which are refined to generate $\TT_*$
and
\begin{align}\label{eq:refined:edge}
 \RR_\ell(\EE_*) := \set{E\in\EE_\ell}{\exists T\in\RR_\ell(\TT_*)\quad E\cap T\neq\emptyset}
\end{align}
be the set of all edges which touch a refined element.
Then,
\begin{align}
\#\RR_\ell(\EE_*) \le \c{refine}\,\#\RR_\ell(\TT_*)
\end{align}
and
\begin{align}\label{eq:dlr}
 \norm{U_*-U_\ell}{H^1(\Omega)}
 \le \c{dlr}\,\varrho_\ell(\RR_\ell(\EE_*))
\end{align}
with constants $\setc{refine},\setc{dlr}>0$ which depend only on $\Omega$, the use of newest
vertex bisection, and the initial mesh $\TT_0$.
\end{proposition}

\begin{proof}
According to shape regularity, the number of elements which share a node $z\in\KK_\ell$ is uniformly bounded. Consequently, so is the number of edges which touch an element $T\in\RR_\ell(\TT_*)$ which will be refined. This proves the estimate $\#\RR_\ell(\EE_*) \le \c{refine}\,\#\RR_\ell(\TT_*)$.
To prove~\eqref{eq:dlr}, we use the discrete local reliability of $\rho_\ell$
from Proposition~\ref{prop:dlr:rho}. With the help of Lemma~\ref{lemma:local}, each refinement indicator $\rho_\ell(T)$ for $T\in\RR_\ell(\TT_*)$ is dominated by finitely many indicators $\varrho_\ell(E)$ for $E\in\RR_\ell(\EE_*)$, where the number depends only on the shape regularity constant $\sigma(\TT_\ell)$.
\end{proof}

%
%

\subsection{Adaptive algorithm based on D\"orfler marking}
\label{section:doerfler}%
\new{Our} version of the adaptive algorithm has been well-studied in the literature mainly for element-based estimators, cf.\ e.g.~\cite{ckns}.

\begin{algorithm}\label{algorithm:doerfler}
Let adaptivity parameter $0<\theta<1$ and initial triangulation $\TT_0$ be given. For each $\ell=0,1,2,\dots$ do:
\begin{itemize}
\item[(i)] Compute discrete solution $U_\ell\in\SS^1(\TT_\ell)$.
\item[(ii)] Compute refinement indicators $\varrho_\ell(E)$ for all $E\in\EE_\ell$.
\item[(iii)] Choose set $\MM_\ell\subseteq\EE_\ell$ with minimal cardinality such that
\begin{align}\label{eq:doerfler}
 \theta\,\varrho_\ell^2 \le \varrho_\ell(\MM_\ell)^2.
\end{align}
\item[(iv)] Generate new mesh $\TT_{\ell+1}:=\refine(\TT_\ell,\MM_\ell)$.
\item[(v)] Update counter $\ell\mapsto\ell+1$ and go to {\rm(i)}.
\end{itemize}
\end{algorithm}

 \subsection{Adaptive algorithm based on modified D\"orfler marking}
 \label{section:becker}%
 For (piecewise) smooth data $f\in H^1$ and $g\in H^2$, uniform mesh-refinement guarantees $\osc{\ell} = \OO(h^2)$ as well as 
 $\oscD{\ell} = \OO(h^{3/2})$, whereas the error and hence the error estimator $\varrho_\ell$ may at most decay as $\OO(h)$. 
 Consequently, we may expect that the normal jump terms dominate the error estimator~\cite{cv}. This observation led to the following 
 version of the marking strategy which has essentially been proposed in~\cite{bms}.
 We stress, however, that the algorithm in~\cite{bms,bm} is stated with node oscillations $\oscK{\ell}$ instead of edge oscillations 
 $\osc{\ell}$. Moreover, certain details in the proofs of~\cite{bm} seem to be dubious.

 \begin{algorithm}\label{algorithm:becker}
 Let adaptivity parameters $0<\theta_1,\theta_2<1$ and $\vartheta>0$ and an initial triangulation $\TT_0$ be given. For each $\ell=0,1,2,\dots$ do:
 \begin{itemize}
 \item[(i)] Compute discrete solution $U_\ell\in\SS^1(\TT_\ell)$.
 \item[(ii)] Compute refinement indicators $\varrho_\ell(E)$ for all $E\in\EE_\ell$.
 \item[(iii.1)] If $\osc{\ell}^2 + \oscD{\ell}^2 \le \vartheta\,\eta_\ell^2$, choose set $\MM_\ell\subseteq\EE_\ell$ with minimal cardinality such that
 \begin{align}\label{eq1:becker}
  \theta_1\,\eta_\ell^2 \le \eta_\ell(\MM_\ell)^2.
 \end{align}
 \item[(iii.2)] Otherwise, choose set $\MM_\ell\subseteq\EE_\ell$ with minimal cardinality such that
 \begin{align}\label{eq2:becker}
  \theta_2\,(\osc{\ell}^2+\oscD{\ell}^2)
  \le \osc{\ell}(\MM_\ell)^2 + \oscD{\ell}(\MM_\ell)^2.
 \end{align}
 \item[(iv)] Generate new mesh $\TT_{\ell+1}:=\refine(\TT_\ell,\MM_\ell)$.
 \item[(v)] Update counter $\ell\mapsto\ell+1$ and go to {\rm(i)}.
 \end{itemize}
 \end{algorithm}

\section{Convergence of Adaptive Algorithm}
\label{section:convergence}%
\noindent
In this section, we prove a contraction property $\Delta_{\ell+1}\le\kappa\,\Delta_\ell$ for some
quasi-error quantity $\Delta_\ell \simeq \varrho_\ell^2$. To that end, we first introduce a locally equivalent error estimator.
 To that end, we first note that the
 modified D\"orfler marking~\eqref{eq1:becker}--\eqref{eq2:becker} implies the D\"orfler
 marking~\eqref{eq:doerfler}.

 \begin{lemma}\label{lemma:doerfler}
 If the set $\MM_\ell\subseteq\EE_\ell$ satisfies the modified D\"orfler marking~\eqref{eq1:becker}--\eqref{eq2:becker} with parameters $0<\theta_1,\theta_2<1$ \and $\vartheta>0$. Then, $\MM_\ell$ satisfies the D\"orfler marking~\eqref{eq:doerfler} with parameter
 $0<\theta := \min\{\theta_1/(1+\vartheta)\,,\,\theta_2/(1+\vartheta^{-1})\}<1$.
 \end{lemma}
 
 \begin{proof}
 In case of $\osc{\ell}^2 + \oscD{\ell}^2 \le \vartheta\,\eta_\ell^2$, it holds that
 $\varrho_\ell^2 \le (1+\vartheta)\,\eta_\ell^2$. This implies
 \begin{align*}
  \frac{\theta_1}{1+\vartheta}\,\varrho_\ell^2 \le \theta_1\,\eta_\ell^2
  \le \eta_\ell(\MM_\ell)^2 \le \varrho_\ell(\MM_\ell)^2.
 \end{align*}
 Otherwise, it holds that $\varrho_\ell^2 \le (1+\vartheta^{-1})\,(\osc{\ell}^2+\oscD{\ell}^2)$ which yields
 \begin{align*}
  \frac{\theta_2}{1+\vartheta^{-1}}\,\varrho_\ell^2
  \le \theta_2\, \big(\osc{\ell}^2 + \oscD{\ell}^2\big)
  \le \big(\osc{\ell}(\MM_\ell)^2+\oscD{\ell}(\MM_\ell)^2\big)
  \le \varrho_\ell(\MM_\ell)^2.
 \end{align*}
 This concludes the proof.
 \end{proof}

\begin{lemma}\label{lemma:boundary}
Let $W_{\ell+1} \in \SS^1(\TT_{\ell+1})$ with
$W_{\ell+1}|_\Gamma = g_{\ell+1}$. Define
$W_{\ell+1}^\ell \in \SS^1(\TT_{\ell+1})$ by
\begin{align}\label{eq:boundary_changed}
 W_{\ell+1}^\ell(z) = \begin{cases}
 W_{\ell+1}(z) & \text{ for } z\in \NN_{\ell+1}\backslash \Gamma,\\
 g_\ell(z) & \text{ for }z \in \NN_{\ell+1}\cap \Gamma.\end{cases}
\end{align}
Then, there holds
\begin{align}\label{eq:stability_boundary}
 \norm{W_{\ell+1} - W_{\ell+1}^\ell}{H^1(\Omega)}
 \le \c{boundary}\,\norm{g_{\ell+1} - g_\ell}{H^{1/2}(\Gamma)},
\end{align}
where $\setc{boundary}>0$ depends only on $\sigma(\TT_\ell)$.\qed
\end{lemma}
\begin{lemma}[equivalent error estimator]\label{lemma:eqivalenterrest}
Consider the extended error estimator
 \begin{align}
  \mu_\ell^{\,2}
  = \sum_{E\in\EE_\ell^\Omega\cup\EE_\ell^N}\eta_\ell(E)^2
  + \sum_{E\in\EE_\ell}\wosc{\ell}(E)^2
  + \sum_{E\in\EE_\ell^D}\oscD{\ell}(E)^2,
 \end{align}
where the oscillation terms $\wosc{\ell}(E)$ read
\begin{align}
 \wosc{\ell}(E)^2 := \begin{cases}
 \osc{\ell}(E)^2&\text{for }E\in\EE_\ell^\Omega,\\
 |T_E|\,\norm{f}{L^2(T_E)}^2
 \quad&\text{for }E\in\EE_\ell^\Gamma
 \text{ and }T\in\TT_\ell\text{ with }E\subset\partial T_E.
 \end{cases}
\end{align}
Then, there holds equivalence in the following sense
\begin{align*}
 \c{equiverrest}^{-1}\, \mu_\ell^{\,2} \leq \varrho_\ell^2 \leq \mu_\ell^{\,2}\quad \text{and} \quad
\varrho_\ell(E) \leq \mu_\ell(E) \text{ for all } E \in \EE_\ell,
\end{align*}
where $\setc{equiverrest}\ge1$ depends only on $\sigma(\TT_\ell)$. Particularly, if $\MM_\ell\subseteq \EE_\ell$ satisfies the D\"orfler marking~\eqref{eq:doerfler}
with $\varrho_\ell$ and $\theta>0$, then $\MM_\ell$ satisfies the D\"orfler
marking with $\mu_\ell$ for some modified parameter
$0<\widetilde\theta:={\theta}/{\c{equiverrest}} <1$.
\end{lemma}

\begin{proof}
The estimates $\varrho_\ell(E) \leq \mu_\ell(E) \text{ for all } E \in \EE_\ell$ are obvious and imply $\varrho_\ell^2\leq \mu_\ell^{\,2}$. The estimate
$\c{equiverrest}^{-1}\, \mu_\ell^{\,2} \leq \varrho_\ell^2$ follows from Lemma~\ref{lemma:local} (ii) \& (iii). Now, we obtain
\begin{align*}
 \widetilde\theta\,\mu_\ell^{\,2} \le \theta\,\varrho_\ell^2
 \le \varrho_\ell(\MM_\ell)^2 \le \mu_\ell(\MM_\ell)^2,
\end{align*}
i.e.\ the estimator $\mu_\ell$ satisfies the D\"orfler marking~\eqref{eq:doerfler}
with $\widetilde\theta:={\theta}/{\c{equiverrest}}$.
\end{proof}

\begin{lemma}[estimator reduction]\label{lemma:reduction}
Assume that the set $\MM_\ell\subseteq\EE_\ell$ of marked edges satisfies the D\"orfler
marking~\eqref{eq:doerfler} with $\varrho_\ell$ and some fixed parameter $0<\theta<1$ and that
$\TT_{\ell+1}=\refine(\TT_\ell,\MM_\ell)$ is obtained by local newest vertex bisection
of $\TT_\ell$.
Then, there holds the estimator reduction estimate
\begin{align}\label{eq:reduction}
\mu_{\ell+1}^{\,2}
\le q\,\mu_\ell^{\,2} + \c{reduction}\norm{\nabla(U_{\ell+1}-U_\ell)}{L^2(\Omega)}^2
\end{align}
with some contraction constant $q\in(0,1)$ which depends only on
$\theta\in(0,1)$. The constant $\setc{reduction}>0$ additionally depends
only on the initial mesh $\TT_0$.
\end{lemma}

\begin{proof}[Sketch of proof]
For the sake of completeness, we include the idea of the proof of~\eqref{eq:reduction}.
To keep the notation simple, we define $\eta_\ell(E) = 0$ for $E\in\EE_\ell^D$ and
$\oscD{\ell}(E) = 0$ for $E\in\EE_\ell^\Omega\cup\EE_\ell^N$ so that all contributions of $\mu_\ell$
are defined on the entire set of edges $\EE_\ell$.%

First, we employ a triangle inequality and the Young inequality to see
\begin{align*}
 \mu_{\ell+1}^{\,2}
 &\le (1+\delta)\,\Big(\sum_{E\in\EE_{\ell+1}^\Omega}|E|\norm{[\partial_nU_\ell]}{L^2(E)}^2
 + \sum_{E\in\EE_{\ell+1}^N}|E|\norm{\phi-\partial_nU_\ell}{L^2(E)}^2\Big)
 \\&\quad
 + (1+\delta^{-1})\Big(
 \sum_{E\in\EE_{\ell+1}^\Omega}|E|\norm{[\partial_n(U_{\ell+1}-U_\ell)]}{L^2(E)}^2
 + \sum_{E\in\EE_{\ell+1}^N}|E|\norm{\partial_n(U_{\ell+1}-U_\ell)}{L^2(E)}^2
 \Big)
 \\&\quad
  + \wosc{\ell+1}^2 + \oscD{\ell+1}^2,
\end{align*}
where $\delta>0$ is arbitrary.
Second, a scaling argument proves
\begin{align*}
 \sum_{E\in\EE_{\ell+1}^\Omega}|E|\norm{[\partial_n(U_{\ell+1}-U_\ell)]}{L^2(E)}^2
 + \sum_{E\in\EE_{\ell+1}^\Gamma}|E|\norm{\partial_n(U_{\ell+1}-U_\ell)}{L^2(E)}^2
 \le C\, \norm{\nabla(U_{\ell+1}-U_\ell)}{L^2(\Omega)}^2,
\end{align*}
and the constant $C>0$ depends only on $\sigma(\TT_\ell)$.
Third, we argue as in~\cite[Corollary 3.4]{ckns} to see
\begin{align*}
 \sum_{E\in\EE_{\ell+1}^\Omega}|E|\norm{[\partial_nU_\ell]}{L^2(E)}^2
 + \sum_{E\in\EE_{\ell+1}^N}|E|\norm{\phi-\partial_nU_\ell}{L^2(E)}^2
 \le \eta_\ell^2 - \frac12\,\eta_\ell(\MM_\ell)^2.
\end{align*}
Fourth, it is part of the proof of~\cite[Theorem 5.4]{agp} that
\begin{align*}
 \oscD{\ell+1}^2 \le \oscD{\ell}^2 - \frac12\,\oscD{\ell}(\MM_\ell)^2,
\end{align*}
which essentially follows from the orthogonality relation~\eqref{eq:nodal:orthogonality}.
Fifth, in~\cite[Lemma 6]{pp} it is proven that
\begin{align}\label{eq:wosc}
 \wosc{\ell+1}^2
 \le \wosc{\ell}^2 - \frac14\,\wosc{\ell}(\MM_\ell)^2.
\end{align}
Plugging everything together, we see
\begin{align*}
 \mu_{\ell+1}^{\,2}
 &\le (1+\delta) \big(\mu_\ell^{\,2} - \frac14\,\mu_\ell(\MM_\ell)^2\big)
 + C(1+\delta^{-1})\,\norm{\nabla(U_{\ell+1}-U_\ell)}{L^2(\Omega)}^2\\
 &\le (1+\delta) (1-\widetilde\theta/4)\mu_\ell^{\,2} + C(1+\delta^{-1})\,\norm{\nabla(U_{\ell+1}-U_\ell)}{L^2(\Omega)}^2,
\end{align*}
where we have used that Lemma~\ref{lemma:eqivalenterrest} guarantees the
 D\"orfler marking for $\mu_\ell$ in the second estimate.
Finally, it only remains to choose $\delta>0$ sufficiently small so that
$q:=(1+\delta)(1-\widetilde\theta/4)<1$.
\end{proof}

The following lemma states some quasi-Galerkin orthogonality property which allows to overcome the lack of Galerkin orthogonality used 
in~\cite{ckns}.

\begin{lemma}[quasi-Galerkin orthogonality]\label{lemma:orthogonality}
Let $\TT_*=\refine(\TT_\ell)$ be an arbitrary refinement of $\TT_\ell$
with the associated Galerkin solution $U_*\in\SS^1(\TT_*)$.
Then,
\begin{align}\label{eq:orthogonality}
 2\,|\dual{\nabla(u-U_*)}{\nabla(U_*-U_{\ell})}_\Omega|
 \le \alpha\,\norm{\nabla(u-U_*)}{L^2(\Omega)}^2
 + \alpha^{-1}\c{orthogonality}\,\norm{h_\ell^{1/2}(g_*-g_\ell)'}{L^2(\Gamma_D)}^2,
\end{align}
for all $\alpha>0$, and consequently
\begin{align}\label{eq:orthogonality1}
\begin{split}
 (1-\alpha)\norm{\nabla(u-U_*)}{L^2(\Omega)}^2
 &\le \norm{\nabla(u-U_\ell)}{L^2(\Omega)}^2
 - \norm{\nabla(U_*-U_\ell)}{L^2(\Omega)}^2
 \\&\qquad
 + \alpha^{-1}\c{orthogonality}\,\norm{h_\ell^{1/2}(g_*-g_\ell)'}{L^2(\Gamma_D)}^2
\end{split}
\end{align}
as well as
\begin{align}\label{eq:orthogonality2}
\begin{split}
 \norm{\nabla(u-U_\ell)}{L^2(\Omega)}^2
 &\le (1+\alpha)\norm{\nabla(u-U_*)}{L^2(\Omega)}^2
 + \norm{\nabla(U_*-U_\ell)}{L^2(\Omega)}^2
 \\&\qquad
 + \alpha^{-1}\c{orthogonality}\,\norm{h_\ell^{1/2}(g_*-g_\ell)'}{L^2(\Gamma_D)}^2.
\end{split}
\end{align}
The constant $\setc{orthogonality}>0$ depends only on the shape regularity of $\sigma(\TT_\ell)$ and $\sigma(\TT_*)$ and on $\Omega$.
\end{lemma}

\begin{proof}
We recall the Galerkin orthogonality
\begin{align*}
 \dual{\nabla(u-U_*)}{\nabla V_*}_\Omega = 0
 \quad \text{for all } V_* \in \SS^1_D(\TT_*).
\end{align*}
Let $U_*^\ell\in\SS^1(\TT_*)$ be the unique Galerkin solution solution of~\eqref{eq:galerkin} with $U_*^\ell|_{\Gamma_D} = g_\ell$. We use the Galerkin orthogonality with $V_*= U_*^\ell - U_\ell \in \SS^1_D(\TT_*)$.
This and the Young inequality allow to estimate the $L^2$-scalar product by
\begin{align*}
 &2\,|\dual{\nabla(u-U_{*})}{\nabla(U_*-U_{\ell})}_\Omega|
 = 2\,|\dual{\nabla(u-U_*)}{\nabla(U_*-U_*^\ell)}_\Omega|
 \\&\qquad
 \le \alpha\,\norm{\nabla(u-U_*)}{L^2(\Omega)}^2
 + \alpha^{-1}\,\norm{\nabla(U_*-U_*^\ell)}{L^2(\Omega)}^2
\end{align*}
for all $\alpha>0$. To estimate the second contribution on the right-hand side, we proceed as in the proof of Proposition~\ref{prop:dlr:rho} and choose
arbitrary extensions $\widehat g_*,\widehat g_\ell\in H^{1/2}(\Gamma)$ of the nodal interpolands $g_*,g_\ell$ from $\Gamma_D$ to $\Gamma$. Then, we use the test function $V_*=(U_*-U_*^\ell) - \LL_*(\widehat g_*-\widehat g_\ell)\in\SS^1_D(\TT_*)$ and the Galerkin orthogonalities for $U_*,U_*^\ell\in\SS^1(\TT_*)$ to see
\begin{align*}
 0
 = \dual{\nabla(u-U_*^\ell)}{\nabla V_*}_\Omega
 - \dual{\nabla(u-U_*)}{\nabla V_*}_\Omega
 = \dual{\nabla(U_*-U_*^\ell)}{\nabla V_*}_\Omega.
\end{align*}
Arguing as above, we obtain
\begin{align}\label{dpr1:open}
 \norm{\nabla(U_*-U_*^\ell)}{L^2(\Omega)}
 \lesssim \norm{g_*-g_\ell}{H^{1/2}(\Gamma_D)}
 \lesssim \norm{h_\ell^{1/2}(g_*-g_\ell)'}{L^2(\Gamma_D)}.
\end{align}
This concludes the proof of~\eqref{eq:orthogonality}.

To verify~\eqref{eq:orthogonality1}--\eqref{eq:orthogonality2}, we
use the identity
\begin{align*}
 &\norm{\nabla(u-U_\ell)}{L^2(\Omega)}^2
 = \norm{\nabla\big((u-U_*)+(U_*-U_\ell)\big)}{L^2(\Omega)}^2
 \\&\qquad
 = \norm{\nabla(u-U_*)}{L^2(\Omega)}^2
 + 2\,\dual{\nabla(u-U_*)}{\nabla(U_*-U_{\ell})}_\Omega
 + \norm{\nabla(U_*-U_{\ell})}{L^2(\Omega)}^2.
\end{align*}
Rearranging the terms accordingly and use of the quasi-Galerkin orthogonality~\eqref{eq:orthogonality} to estimate the scalar product, concludes the proof.
\end{proof}

\begin{theorem}[contraction of quasi-error]\label{thm:contraction}
For \new{the adaptive algorithm stated in Algorithm~\ref{algorithm:doerfler}} above,
there are constants $\gamma,\lambda>0$ and $0<\kappa<1$ such that the combined error quantity
\begin{align}\label{eq:delta}
 \Delta_\ell
 := \norm{\nabla(u - U_\ell)}{L^2(\Omega)}^2
 + \lambda\,\oscD{\ell}^2 + \gamma\,\mu_\ell^{\,2}
 \ge 0
\end{align}
satisfies a contraction property
\begin{align}\label{eq:contraction}
 \Delta_{\ell+1} \le \kappa\,\Delta_\ell
 \quad\text{for all }\ell\in\N_0.
\end{align}
In particular, this implies $\lim\limits_{\ell\to\infty}\varrho_\ell = 0 = \lim\limits_{\ell\to\infty}\norm{u - U_\ell}{H^1(\Omega)}$.
\end{theorem}

\begin{proof}
\new{Using the quasi-Galerkin orthogonality~\eqref{eq:orthogonality1} with $\TT_*=\TT_{\ell+1}$, we see}
\begin{align*}
 (1-\alpha)\,\norm{\nabla(u-U_{\ell+1})}{L^2(\Omega)}^2
 &\le \norm{\nabla(u-U_{\ell})}{L^2(\Omega)}^2
 - \norm{\nabla(U_{\ell+1}-U_{\ell})}{L^2(\Omega)}^2
 \\&\qquad
 + \alpha^{-1}\c{orthogonality}\,\norm{h_\ell^{1/2}(g_{\ell+1}-g_\ell)'}{L^2(\Gamma_D)}^2.
\end{align*}
The orthogonality relation~\eqref{eq:nodal:orthogonality} applied for $g_{\ell+1}\in\SS^1(\TT_{\ell+1}|_{\Gamma_D})$ yields
\begin{align*}
 \oscD{\ell+1}^2 + \norm{h_\ell^{1/2}(g_{\ell+1}-g_\ell)'}{L^2(\Gamma_D)}^2
 \le \norm{h_\ell^{1/2}(g-g_\ell)'}{L^2(\Gamma_D)}^2
 = \oscD{\ell}^2.
\end{align*}
Together with the aforegoing estimate, we obtain
\begin{align*}
 &(1-\alpha)\,\norm{\nabla(u-U_{\ell+1})}{L^2(\Omega)}^2
 + \alpha^{-1}\c{orthogonality}\,\oscD{\ell+1}^2
 \\&\qquad
 \le
 \norm{\nabla(u-U_{\ell})}{L^2(\Omega)}^2
 + \alpha^{-1}\c{orthogonality}\,\oscD{\ell}^2
 - \norm{\nabla(U_{\ell+1}-U_\ell)}{L^2(\Omega)}^2.
\end{align*}
We add the error estimator and use the estimator reduction~\eqref{eq:reduction} to see, for $\beta>0$,
\begin{align*}
 &(1-\alpha)\,\norm{\nabla(u-U_{\ell+1})}{L^2(\Omega)}^2
 + \alpha^{-1}\c{orthogonality}\,\oscD{\ell+1}^2
 + \beta\,\mu_{\ell+1}^{\,2}
 \\&\qquad
 \le \norm{\nabla(u-U_{\ell})}{L^2(\Omega)}^2
 + \alpha^{-1}\c{orthogonality}\,\oscD{\ell}^2
 + \beta\, q\,\mu_\ell^{\,2}
 + (\beta\c{reduction}-1) \, \norm{\nabla(U_{\ell+1}-U_\ell)}{L^2(\Omega)}^2.
\end{align*}
We choose $\beta>0$ sufficiently small to guarantee $\beta\c{reduction}-1\le0$. Then, we use the
reliability~\eqref{eq:reliable} of $\varrho_\ell\le\mu_\ell$ in the form
\begin{align*}
 \c{reliable}^{-1}\,\norm{\nabla(u-U_\ell)}{L^2(\Omega)}
 \le \c{reliable}^{-1}\,\norm{u-U_\ell}{H^1(\Omega)}
 \le \mu_\ell
\end{align*}
to see, for $\eps>0$,
\begin{align*}
 &(1-\alpha)\,\norm{\nabla(u-U_{\ell+1})}{L^2(\Omega)}^2
 + \alpha^{-1}\c{orthogonality}\,\oscD{\ell+1}^2
 + \beta\,\mu_{\ell+1}^{\,2}
 \\&\qquad
 \le (1-\eps\beta\c{reliable}^{-2})\,\norm{\nabla(u-U_{\ell})}{L^2(\Omega)}^2
 + \alpha^{-1}\c{orthogonality}\,\oscD{\ell}^2
 + \beta (q+\eps)\,\mu_\ell^{\,2}.
\end{align*}
Moreover, since $\oscD{\ell}$ is a contribution of $\mu_\ell$, we have $\oscD{\ell}\le\mu_\ell$, whence, for $\delta>0$,
\begin{align*}
 &(1-\alpha)\,\norm{\nabla(u-U_{\ell+1})}{L^2(\Omega)}^2
 + \alpha^{-1}\c{orthogonality}\,\oscD{\ell+1}^2
 + \beta\,\mu_{\ell+1}^{\,2}
 \\&\qquad
 \le (1-\eps\beta\c{reliable}^{-2})\,\norm{\nabla(u-U_{\ell})}{L^2(\Omega)}^2
 + (1-\delta\beta)\,\alpha^{-1}\c{orthogonality}\,\oscD{\ell}^2
 + \beta (q+\eps+\delta\,\alpha^{-1}\c{orthogonality})\,\mu_\ell^{\,2}.
\end{align*}
For $0<\alpha<1$, we may now rearrange this estimate to end up with
\begin{align*}
 &\norm{\nabla(u-U_{\ell+1})}{L^2(\Omega)}^2
 + \frac{\c{orthogonality}}{\alpha(1-\alpha)}\,\oscD{\ell+1}^2
 + \frac{\beta}{1-\alpha}\,\mu_{\ell+1}^{\,2}
 \\&\qquad
 \le \frac{1-\eps\beta\c{reliable}^{-2}}{1-\alpha}\,\norm{\nabla(u-U_{\ell})}{L^2(\Omega)}^2
 + (1-\delta\beta)\,\frac{\c{orthogonality}}{\alpha(1-\alpha)}\,\oscD{\ell}^2
 \\&\qquad\qquad
 + (q+\eps+\delta\,\alpha^{-1}\c{orthogonality})\,\frac{\beta}{1-\alpha}\,\mu_\ell^{\,2}.
\end{align*}
It remains to choose the free constants $0<\alpha,\delta,\eps<1$, whereas $\beta>0$ has already been fixed:
\begin{itemize}
\item First, choose $0<\eps<\c{reliable}^2/\beta$ sufficiently small to guarantee $0<q+\eps<1$.
\item Second, choose $0<\alpha<1$ sufficiently small such that $0<(1-\eps\beta\c{reliable}^{-2})/(1-\alpha)<1$.
\item Third, choose $\delta>0$ sufficiently small with $0<q+\eps+\delta\,\alpha^{-1}\c{orthogonality}<1$.
\end{itemize}
With $\gamma:=\beta/(1-\alpha)$, $\lambda:=\alpha^{-1}\,\c{orthogonality}/(1-\alpha)$, and
$0<\kappa<1$ the maximal contraction constant of the three contributions,
we conclude the proof of~\eqref{eq:contraction}.
\end{proof}

\section{Quasi-Optimality of Adaptive Algorithm}
\label{section:optimal}%
\vspace*{-3mm}
\subsection{Optimality of marking strategy}
\label{section:optimal:doerfler}%
\new{With Theorem~\ref{thm:contraction}, we have seen that D\"orfler marking~\eqref{eq:doerfler} yields a contraction of $\Delta_\ell \simeq \varrho_\ell^2$.} 
In the following, we first observe that the D\"orfler marking~\eqref{eq:doerfler} \new{is not only sufficient but in some sense also necessary to obtain contraction of 
the estimator.}

\begin{proposition}[optimality of D\"orfler marking]
\label{prop:doerfler}
Let $\alpha>0$ and assume that the adaptivity parameter $0<\theta<1$ is sufficiently small, more precisely
\begin{align}\label{eq:doerfler:theta}
 q_\star := \frac{1-\theta(\c{dlr}^2+1+\alpha^{-1}\c{orthogonality})\c{efficient}^2}{1+\alpha}>0.
\end{align}
Let $0<q\le q_\star$ and $\TT_*=\refine(\TT_\ell)$ and assume that
\begin{align}\label{eq:doerfler:contraction}
\begin{split}
 &\big(\norm{\nabla(u-U_*)}{L^2(\Omega)}^2 + \osc{*}^2 + \oscD{*}^2 + \oscN{*}^2\big)
 \\&\qquad
 \le q\,\big(\norm{\nabla(u-U_\ell)}{L^2(\Omega)}^2 + \osc{\ell}^2 + \oscD{\ell}^2 + \oscN{\ell}^2\big).
\end{split}
\end{align}
Then, there holds the D\"orfler marking for the set $\RR_\ell(\EE_*)\subseteq\EE_\ell$ defined in~\eqref{eq:refined:edge}, i.e.\
\begin{align}\label{eq:doerfler:refined}
 \theta\,\varrho_\ell^2
 \le \varrho_\ell(\RR_\ell(\EE_*))^2.
\end{align}
\end{proposition}

\begin{proof}
We start with the elementary observation that $q\le q_\star$ is equivalent to
\begin{align*}
 \theta \le \frac{1-q(1+\alpha)}{(\c{dlr}^2+1+\alpha^{-1}\c{orthogonality})\c{efficient}^2}.
\end{align*}
Using the discrete local reliability~\eqref{eq:dlr} and the quasi-Galerkin orthogonality~\eqref{eq:orthogonality2}, we see
\begin{align*}
 \c{dlr}^2&\varrho_\ell(\RR_\ell(\EE_*))^2
 \ge \norm{\nabla(U_*-U_\ell)}{L^2(\Omega)}^2
 \\&
 \ge \norm{\nabla(u-U_\ell)}{L^2(\Omega)}^2 - (1+\alpha)\,\norm{\nabla(u-U_*)}{L^2(\Omega)}^2
 - \alpha^{-1}\c{orthogonality}\,\norm{h_\ell^{1/2}(g_*-g_\ell)'}{L^2(\Gamma_D)}^2
 \\&
 = \big(\norm{\nabla(u-U_\ell)}{L^2(\Omega)}^2 + \osc{\ell}^2 + \oscD{\ell}^2+\oscN{\ell}^2\big)
 \\&\quad
 - (1+\alpha)\big(\norm{\nabla(u-U_*)}{L^2(\Omega)}^2 + \osc{*}^2 + \oscD{*}^2+\oscN{*}^2\big)
 \\&\quad
 - \osc{\ell}^2 - \oscD{\ell}^2-\oscN{\ell}^2 + (1+\alpha)(\osc{*}^2 + \oscD{*}^2 +\oscN{*}^2)
 \\&\quad
 - \alpha^{-1}\c{orthogonality}\,\norm{h_\ell^{1/2}(g_*-g_\ell)'}{L^2(\Gamma_D)}^2
 \\&\ge
 \big(1-q(1+\alpha)\big)\big(\norm{\nabla(u-U_\ell)}{L^2(\Omega)}^2 + \osc{\ell}^2 + \oscD{\ell}^2+\oscN{\ell}^2\big)
 \\&\quad
 - \osc{\ell}^2 - \oscD{\ell}^2 -\oscN{\ell}^2+ (1+\alpha)(\osc{*}^2 +\oscD{*}^2+\oscN{*}^2)
 \\&\quad
 - \alpha^{-1}\c{orthogonality}\,\norm{h_\ell^{1/2}(g_*-g_\ell)'}{L^2(\Gamma_D)}^2,
\end{align*}
where we have finally used Assumption~\eqref{eq:doerfler:contraction}. As in the proof of Proposition~\ref{prop:dlr:rho}, we have
\begin{align*}
 \norm{h_\ell^{1/2}(g_*-g_\ell)'}{L^2(\Gamma_D)}^2
 \leq \oscD{\ell}(\RR_\ell(\EE_*))^2
 \le \varrho_\ell(\RR_\ell(\EE_*))^2.
\end{align*}
Moreover, the identities $\oscD{\ell}(E) = \oscD{*}(E)$, $\osc{\ell}(E) = \osc{*}(E)$ and $\oscN{\ell}(E)=\oscN{*}(E)$ for $E\in\EE_\ell\backslash\RR_\ell(\EE_*)$ prove
\begin{align}\label{eq1:identity}
 \oscD{\ell}^2 - \oscD{*}^2
 &\le \oscD{\ell}(\RR_\ell(\EE_*))^2,\\
 \label{eq2:identity}
 \osc{\ell}^2 - \osc{*}^2
 &\le \osc{\ell}(\RR_\ell(\EE_*))^2,\\
 \label{eq3:identity}
 \oscN{\ell}^2 - \oscN{*}^2
 &\le \oscN{\ell}(\RR_\ell(\EE_*))^2.
\end{align}
Note that~\eqref{eq2:identity} led to the definition of $\RR_\ell(\EE_*)$ given above. Together with the efficiency~\eqref{eq:efficient}
and $\oscD{\ell}(\RR_\ell(\EE_*))^2 + \osc{\ell}(\RR_\ell(\EE_*))^2 +\oscN{\ell}(\RR_\ell(\EE_*))^2\le \varrho_\ell(\RR_\ell(\EE_*))^2$, we may now conclude
\begin{align*}
 \big(\c{dlr}^2 + 1 + \alpha^{-1}\c{orthogonality}\big)\,\varrho_\ell(\RR_\ell(\EE_*))^2
 \ge \big(1-q(1+\alpha)\big)\,\c{efficient}^{-2}\,\varrho_\ell^2.
\end{align*}
This is equivalent to $\theta\,\varrho_\ell^2 \le \varrho_\ell(\RR_\ell(\EE_*))^2$ and led to the definition of $q_\star$.
\end{proof}

 \begin{proposition}[optimality of modified D\"orfler marking]
 \label{prop:becker}
 Let $\alpha>0$ and $0<\theta_2<1$ and assume that the adaptivity parameters $0<\theta_1,\vartheta<1$ are sufficiently small, more precisely
 \begin{align}\label{eq:becker:theta}
  q_\star:=\max\Big\{\frac{1-\c{efficient}^2\big(\theta_1(1\!+\!\c{dlr}^2) + \vartheta(1\!+\!\c{dlr}^2\!+\!\alpha^{-1}\c{orthogonality})\big)}{1+\alpha}
  \,,\,
  \frac{1-\theta_2}{(1+\vartheta^{-1})(\c{reliable}^2+1)}\Big\} >0.
 \end{align}
 Let $0<q\le q_\star$ and $\TT_*=\refine(\TT_\ell)$
 and assume that
 \begin{align}\label{eq:becker:contraction}
 \begin{split}
  &\big(\norm{\nabla(u-U_*)}{L^2(\Omega)}^2 + \osc{*}^2 + \oscD{*}^2 + \oscN{*}^2\big)
  \\&\qquad
  \le q\,\big(\norm{\nabla(u-U_\ell)}{L^2(\Omega)}^2 + \osc{\ell}^2 + \oscD{\ell}^2 + \oscN{\ell}^2\big).
 \end{split}
 \end{align}
 Then, there holds the modified D\"orfler marking for the set $\RR_\ell(\EE_*)\subseteq\EE_\ell$, i.e.\ there holds either
 \begin{align}\label{eq1:becker:refined}
  \theta_1\,\eta_\ell^2
  \le \eta_\ell(\RR_\ell(\EE_*))^2
 \end{align}
 in case of $\osc{\ell}^2 + \oscD{\ell}^2 \le \vartheta\,\eta_\ell^2$ or
 \begin{align}\label{eq2:becker:refined}
  \theta_2\,\big(\osc{\ell}^2 + \oscD{\ell}^2\big)
  \le \osc{\ell}(\RR_\ell(\EE_*))^2 + \oscD{\ell}(\RR_\ell(\EE_*))^2
 \end{align}
 otherwise.
 \end{proposition}
 
 \begin{proof}
 We first assume $\osc{\ell}^2 + \oscD{\ell}^2 \le \vartheta\,\eta_\ell^2$. Arguing as in the proof of Proposition~\ref{prop:doerfler}, we see
 \begin{align*}
  \c{dlr}^2\varrho_\ell(\RR_\ell(\EE_*))^2
  &\ge \norm{\nabla(U_*-U_\ell)}{L^2(\Omega)}^2
  \\&
  \ge \big(1-q(1+\alpha)\big)\big(\norm{\nabla(u-U_\ell)}{L^2(\Omega)}^2 + \osc{\ell}^2 + \oscD{\ell}^2+\oscN{\ell}^2\big)
  \\&\quad
  - \osc{\ell}^2 - \oscD{\ell}^2 -\oscN{\ell}^2+ (1+\alpha)(\osc{*}^2 + \oscD{*}^2+\oscN{*}^2)
  \\&\quad
  - \alpha^{-1}\c{orthogonality}\,\norm{h_\ell^{1/2}(g_*-g_\ell)'}{L^2(\Gamma_D)}^2
  \\&
  \ge \big(1-q(1+\alpha)\big)\,\c{efficient}^{-2}\,\eta_\ell^2
  - \osc{\ell}^2 - \oscD{\ell}^2-\oscN{\ell}(\RR_\ell(\EE_*))^2
  \\&\quad
  - \alpha^{-1}\c{orthogonality}\,\norm{h_\ell^{1/2}(g_*-g_\ell)'}{L^2(\Gamma_D)}^2,
 \end{align*}
 where we have used~\eqref{eq3:identity}.
 Next, we recall the edge-wise definition $\varrho_\ell^2 = \eta_\ell^2 + \osc{\ell}^2 + \oscD{\ell}^2$ and collect all oscillation terms on the right-hand side. Together with $\norm{h_\ell^{1/2}(g_*-g_\ell)'}{L^2(\Gamma_D)}\le\oscD{\ell}$ and $\oscN{\ell}(E)\le\eta_\ell(E)$, this leads to
 \begin{align*}
  \c{dlr}^2\,\eta_\ell(\RR_\ell(\EE_*))^2
  &\ge \big(1-q(1+\alpha)\big)\,\c{efficient}^{-2}\,\eta_\ell^2
  - (1+\c{dlr}^2)\osc{\ell}^2\\
 &\quad - (1+\c{dlr}^2+\alpha^{-1}\c{orthogonality})\oscD{\ell}^2-\oscN\ell(\RR_\ell(\EE_*))^2
  \\&
  \ge \big[\big(1-q(1+\alpha)\big)\,\c{efficient}^{-2}-\vartheta(1+\c{dlr}^2+\alpha^{-1}\c{orthogonality})\big]\,\eta_\ell^2-\eta_\ell(\RR_\ell(\EE_*))^2.
 \end{align*}
 We then conclude
 \begin{align*}
  \eta_\ell(\RR_\ell(\EE_*))^2
  \ge \frac{1-q(1+\alpha)-\vartheta\c{efficient}^2(1+\c{dlr}^2+\alpha^{-1}\c{orthogonality})}{(1+\c{dlr}^2)\c{efficient}^2}\,\eta_\ell^2
  \ge \theta_1\,\eta_\ell^2,
 \end{align*}
 which follows from our assumption on $0<q \le q_\star<1$ and the definition of $q_\star$ in~\eqref{eq:becker:theta}. This concludes the proof of~\eqref{eq1:becker:refined}.
 
 Second, we assume $\osc{\ell}^2 + \oscD{\ell}^2 > \vartheta\,\eta_\ell^2$.
 Recall the estimates~\eqref{eq1:identity}--\eqref{eq3:identity}.
 Then, reliabi\-lity~\eqref{eq:reliable} of $\varrho_\ell^2=\eta_\ell^2+\osc{\ell}^2+\oscD{\ell}^2$ and $\oscN{\ell}\leq \eta_\ell$ yield
 \begin{align*}
  \big(\osc{\ell}^2+\oscD{\ell}^2\big) - \big(\osc{\ell}(\RR_\ell(\EE_*))^2 &
 + \oscD{\ell}(\RR_\ell(\EE_*))^2\big)\\
  &\le \osc{*}^2 + \oscD{*}^2
  \\&
  \le q\,\big(\norm{\nabla(u-U_\ell)}{L^2(\Omega)}^2 + \osc{\ell}^2 + \oscD{\ell}^2 +\oscN{\ell}^2\big)
  \\ &
  \le q\,\big((\c{reliable}^2+1)\,(\eta_\ell^2+\osc{\ell}^2 + \oscD{\ell}^2)\big)
  \\ &
  < q\,(1+\vartheta^{-1})(\c{reliable}^2+1)\,(\osc{\ell}^2 + \oscD{\ell}^2).
 \end{align*}
 Rearranging the terms, we obtain
 \begin{align*}
 \theta_2\,\big(\osc{\ell}^2 + \oscD{\ell}^2\big)
 &\le \big[1-q\,(1+\vartheta^{-1})(\c{reliable}^2+1)\big]\,\big(\osc{\ell}^2 + \oscD{\ell}^2\big)
 \\
 &\le \osc{\ell}(\RR_\ell(\EE_*))^2 + \oscD{\ell}(\RR_\ell(\EE_*))^2,
 \end{align*}
 where the first estimate follows from $0<q\le q_\star<1$ and the definition of $q_\star$ in~\eqref{eq:becker:theta}.
 \end{proof}

\subsection{Optimality of newest vertex bisection}
\label{section:optimal:nvb}
The quasi-optimality analysis for adaptive FEM involves two properties of the mesh-refinement which are, so far, only mathematically guaranteed for newest vertex bisection
\cite{bdd,kpp,ks,stevenson:nvb} and local red-refinement
with hanging nodes up to some fixed order~\cite{bn}.

\dpr{First,} it has originally been proven in~\cite{bdd} and lateron improved in~\cite{stevenson:nvb,ks,kpp} that
the sequence of meshes defined inductively by $\TT_{\ell+1}:=\refine(\TT_\ell,\MM_\ell)$ with arbitrary $\MM_\ell\subseteq\EE_\ell$ satisfies
\begin{align}\label{eq1:nvb}
 \#\TT_\ell - \#\TT_0 \le \c{nvb}\,\sum_{j=0}^{\ell-1}\#\MM_j
 \quad\text{for all }\ell\in\N
\end{align}
with some constant $\setc{nvb}>0$ which depends only on $\TT_0$.
This proves that the closure step in newest vertex bisection which avoids hanging nodes and leads to possible bisections of edges 
$E\in\EE_\ell\backslash\MM_\ell$ may not lead to arbitrary many refinements. 
For newest vertex bisection, the original analysis of~\cite{bdd} as well as of the successors~\cite{ks,stevenson:nvb} required that the reference edges of the initial mesh $\TT_0$ are chosen such that an interior edge $E=T_+\cap T_-\in\EE_0^\Omega$ is either the reference edge of both elements $T_+,T_-\in\TT_0$ or of none. For the particular 2D situation, the recent work~\cite{kpp} removes any assumption on $\TT_0$.

Second, for two meshes $\TT'=\refine(\TT_0)$ and $\TT''=\refine(\TT_0)$ obtained by newest vertex bisection of the initial mesh $\TT_0$, there is a unique coarsest common refinement $\TT'\oplus\TT'' = \refine(\TT_0)$ which is a refinement of both $\TT'$ and $\TT''$. It is shown in~\cite{stevenson,ckns} that $\TT'\oplus\TT''$ is, in fact, the overlay of these meshes. Moreover, it holds that
\begin{align}\label{eq2:nvb}
 \#(\TT'\oplus\TT'') \le \#\TT' + \#\TT'' - \#\TT_0.
\end{align}

\subsection{Definition of approximation class}
\noindent
To state the optimality result, we have to introduce the appropriate approximation
class. Let
\begin{align}
 \T := \set{\TT}{\TT = \refine(\TT_0)}
\end{align}
be the set of all triangulations which can be obtained from $\TT_0$ by newest
vertex bisection. Moreover, let
\begin{align}
 \T_N := \set{\TT\in\T}{\#\TT-\#\TT_0\le N}
\end{align}
be the set of triangulations which have at most $N\in\N$ elements more than the
initial mesh $\TT_0$. For $s>0$, the approximation class $\A_s$
has already been defined in~\eqref{eq:optimal:class}--\eqref{eq:optimal:norm}. The first step is to prove that, up to constants, nodal interpolation of the boundary data yields the best possible approximation of the exact solution.

\begin{lemma}
\label{lem:quasiopt}
The Galerkin solution $U_\ell \in \SS^1(\TT_\ell)$ of~\eqref{eq:galerkin} satisfies
\begin{align}
\begin{split}
&\norm{\nabla(u-U_\ell)}{L^2(\Omega)}^2 + \new{\oscD{\ell}^2} 
\leq \c{szoptimality}\big (\inf_{W_\ell \in \SS^1(\TT_\ell)}
\norm{\nabla(u-W_\ell)}{L^2(\Omega)}^2 \!+\! \new{\oscD{\ell}^2} \big),
\end{split}
\end{align}
where $\setc{szoptimality}>0$ depends only on $\Gamma$ and $\sigma(\TT_\ell)$.
\end{lemma}

\begin{proof}
Let $\widehat g,\widehat g_\ell\in H^{1/2}(\Gamma)$ denote arbitrary
extensions of $g=u|_{\Gamma_D}$ resp.\ $g_\ell$.
Note that $(\LL_\ell P_\ell \widehat g)|_{\Gamma_D}=(P_\ell u)|_{\Gamma_D}$ as well as
$(\LL_\ell P_\ell \widehat g_\ell)|_{\Gamma_D}=g_\ell$, where $\LL_\ell$ denotes the discrete lifting operator from~\eqref{eq:discreteLifting}.
For $V_\ell\in\SS^1_D(\TT_\ell)$, we thus have
$U_\ell - (V_\ell + \LL_\ell P_\ell\widehat g_\ell)\in\SS^1_D(\TT_\ell)$, whence
\begin{align*}
 \norm{\nabla(u-U_\ell)}{L^2(\Omega)}^2
 = \dual{\nabla(u-U_\ell)}{\nabla(u-(V_\ell + \LL_\ell P_\ell\widehat g_\ell))}_\Omega
\end{align*}
according to the Galerkin orthogonality. Therefore, the Cauchy-Schwarz inequality provides the C\'ea-type quasi-optimality
\begin{align*}
 \norm{\nabla(u-U_\ell)}{L^2(\Omega)}
 \le\min_{V_\ell\in\SS^1_D(\TT_\ell)}\norm{\nabla(u-(V_\ell + \LL_\ell P_\ell\widehat g_\ell))}{L^2(\Omega)}.
\end{align*}
We now plug-in $V_\ell = P_\ell u-\LL_\ell P_\ell \widehat g\in\SS^1_D(\TT_\ell)$ to see
\begin{align*}
 \norm{\nabla(u-U_\ell)}{L^2(\Omega)}
 &\le \norm{\nabla(u-P_\ell u + \LL_\ell P_\ell(\widehat g-\widehat g_\ell))}{L^2(\Omega)}
 \\&
 \lesssim \norm{\nabla(u-P_\ell u)}{L^2(\Omega)}
 + \norm{\widehat g-\widehat g_\ell}{H^{1/2}(\Gamma)}.
\end{align*}
Since the extensions $\widehat g,\widehat g_\ell$ of $g,g_\ell$ were arbitrary, we obtain
\begin{align*}
 \norm{\nabla(u-U_\ell)}{L^2(\Omega)}
 &\lesssim \norm{\nabla(u-P_\ell u)}{L^2(\Omega)}
 + \norm{g-g_\ell}{H^{1/2}(\Gamma_D)}\\
 &\lesssim \min_{W_\ell\in\SS^1(\TT_\ell)}\norm{\nabla(u-W_\ell)}{L^2(\Omega)}
 + \norm{h_\ell^{1/2}(g-g_\ell)'}{L^2(\Gamma_D)}\\
&\new{= \min_{W_\ell\in\SS^1(\TT_\ell)}\norm{\nabla(u-W_\ell)}{L^2(\Omega)} + \oscD{\ell}.}
\end{align*}
where we have used the quasi-optimality of the Scott-Zhang projection, see Section~\ref{section:sz}, and Lemma~\ref{lemma:apx}.
\new{Adding $\oscD{\ell}$ to this estimate, we conclude the proof.}
\end{proof}

\subsection{Quasi-optimality result}
Finally, we may formally state the optimality result~\eqref{eq:optimal:order}
described in the introduction.

\begin{theorem}\label{thm:quasioptimality}
Suppose that the adaptivity parameter $0<\theta<1$ in Algorithm~\ref{algorithm:doerfler} satisfies~\eqref{eq:doerfler:theta} \new{so 
that the} marking strategy is optimal in the sense of \new{Proposition~\ref{prop:doerfler}}. 
Let $U_\ell\in\SS^1(\TT_\ell)$ denote the sequence of discrete solutions generated by 
\new{Algorithm~\ref{algorithm:doerfler}}. If the given data and the corresponding weak solution 
of~\eqref{eq:weakform} satisfy $(u,f,g,\phi)\in\A_s$, there holds
\begin{align}
 \norm{u-U_\ell}{H^1(\Omega)}
 \le\c{afem}(\#\TT_\ell-\#\TT_0)^{-s},
\end{align}
i.e.\
each possible convergence rate $s>0$ is asymptotically achieved by AFEM. The constant $\setc{afem}>0$ depends only on $\norm{(u,f,g,\phi)}{\A_s}$, the initial mesh $\TT_0$, and the adaptivity parameters.
\end{theorem}

\begin{proof}
Since the proof follows essentially the lines of~\cite{stevenson,ckns}, we leave the elaborate details to the reader. For any $\eps>0$, the definition of the approximation class $\A_s$ guarantees some triangulation $\TT_\eps\in\T$ such that
\begin{align*}
 \inf_{W_\eps \in \SS^1(\TT_\eps)}
 \big(\norm{\nabla(u-W_\eps)}{L^2(\Omega)}^2 + \normL2{h_\eps^{1/2}(g-W_\eps|_\Gamma)^\prime}{\Gamma_D}^2
 + \oscT{\eps}^2 + \oscN{\eps}^2\big)^{1/2}
 \le\eps
\end{align*}
and
\begin{align*}
 \#\TT_\eps -\#\TT_0 \lesssim \eps^{-1/s},
\end{align*}
where the constant depends only on $\norm{(u,f,g,\phi)}{\A_s}$.
We now consider the overlay $\TT_*:=\TT_\eps\oplus\TT_\ell$. With the help of Lemma~\ref{lem:quasiopt} as well as the elementary estimates $\oscT{*}\le\oscT{\eps}$ and $\oscN{*}\le\oscN{\eps}$, we observe
\begin{align*}
 \Lambda_*:=\big(\norm{\nabla(u-U_*)}{L^2(\Omega)}^2 + \oscD{*}^2
+ \oscT{*}^2+\oscN{*}^2\big)^{1/2}
\lesssim\eps,
\end{align*}
since $\SS^1(\TT_\eps)\subseteq\SS^1(\TT_*)$.
Moreover, the overlay estimate~\eqref{eq2:nvb} predicts
\begin{align*}
 \#\RR_\ell(\TT_*)
 \le \#\TT_*-\#\TT_\ell
 \le \#\TT_\eps-\#\TT_0
 \lesssim \eps^{-1/s}.
\end{align*}
Note that Lemma~\ref{lemma:local} together with reliability and efficiency of $\varrho_*$ yield
\begin{align*}
 \Lambda_* \simeq \big(\norm{\nabla(u-U_*)}{L^2(\Omega)}^2
+ \osc{*}^2 + \oscD{*}^2+\oscN{*}^2\big)^{1/2},
\end{align*}
where $\oscT{*}$ is replaced by $\osc{*}$.
Choosing $\eps = \lambda\big(\norm{\nabla(u-U_\ell)}{L^2(\Omega)}^2 + \oscD{\ell}^2
+ \osc{\ell}^2+\oscN{\ell}^2\big)^{1/2}$ with $\lambda>0$ sufficiently small, we enforce \new{the 
reduction~\eqref{eq:doerfler:contraction}} and derive that 
$\RR_\ell(\EE_*)\subseteq\EE_\ell$ \new{satisfies the D\"orfler marking} criterion, \new{cf.\ Proposition~\ref{prop:doerfler}} 
. Minimality of $\MM_\ell$ thus gives
\begin{align*}
 \#\MM_\ell \le \#\RR_\ell(\EE_*) \lesssim \#\RR_\ell(\TT_*)
 \lesssim \eps^{-1/s}
 \simeq \big(\norm{\nabla(u-U_\ell)}{L^2(\Omega)}^2
+ \osc{\ell}^2 + \oscD{\ell}^2+\oscN{\ell}^2\big)^{-1/(2s)}.
\end{align*}
We next note that
\begin{align*}
 \varrho_\ell^2
 \simeq \norm{\nabla(u-U_\ell)}{L^2(\Omega)}^2
 + \osc{\ell}^2+\oscD{\ell}^2+\oscN{\ell}^2
 \simeq \Delta_\ell
\end{align*}
according to reliability and efficiency of $\varrho_\ell$ and the definition of the contraction quantity $\Delta_\ell$ in Theorem~\ref{thm:contraction}.
Combining the last two lines, we see
\begin{align*}
 \#\MM_\ell \lesssim \Delta_\ell^{-1/(2s)} \simeq \varrho_\ell^{-1/s}
 \quad\text{for all }\ell\in\N_0.
\end{align*}
By use of the closure estimate~\eqref{eq1:nvb} of newest vertex bisection, we obtain
\begin{align*}
 \#\TT_\ell - \#\TT_0
 \lesssim \sum_{j=0}^{\ell-1}\#\MM_j
 \lesssim \sum_{j=0}^{\ell-1}\Delta_j^{-1/(2s)}.
\end{align*}
Note that the contraction property~\eqref{eq:contraction} of $\Delta_j$ implies $\Delta_\ell \le \kappa^{\ell-j}\Delta_j$, whence
$\Delta_j^{-1/(2s)} \le \kappa^{(\ell-j)/(2s)}\Delta_\ell^{-1/(2s)}$. According to $0<\kappa<1$ and the geometric series, this gives
\begin{align*}
 \#\TT_\ell - \#\TT_0
 \lesssim \Delta_\ell^{-1/(2s)} \sum_{j=0}^{\ell-1}\kappa^{(\ell-j)/(2s)}
 \lesssim \Delta_\ell^{-1/(2s)} \simeq \varrho_\ell^{-1/s}.
\end{align*}
Altogether, we may therefore conclude $\norm{u-U_\ell}{H^1(\Omega)}\lesssim\varrho_\ell \lesssim (\#\TT_\ell - \#\TT_0)^{-s}$.
\end{proof}

 \begin{remark}
 All convergence and optimality results in this paper are stated for the edge-based error estimator $\varrho_\ell$. Nevertheless, it is 
 only a notational modification to see that also the element-based error estimator $\rho_\ell$ 
 from~\eqref{eq1:estimator:T}--\eqref{eq2:estimator:T} leads to quasi-optimally convergent versions of AFEM. To that end, 
 \new{Algorithm~\ref{algorithm:doerfler} is} slightly modified, and one seeks minimial sets of marked 
 elements $\MM_\ell\subseteq\TT_\ell$ instead. For each marked element $T\in\MM_\ell$, we mark its reference edge. The convergence 
 result in Theorem~\ref{thm:contraction} and the optimality result in Theorem~\ref{thm:quasioptimality} hold accordingly.
 \end{remark}

\section{Some Remarks on the 3D Case}
\label{section:remarks3d}%
\noindent
So far, we have only considered a 2D model problem~\eqref{eq:strongform}. In 3D, one additional difficulty is that the regularity 
assumption $g\in H^1(\Gamma_D)$ is not sufficient to guarantee continuity of $g$. Therefore, one must not use
nodal interpolation to discretize $g\approx g_\ell$ and to define the Dirichlet data oscillations $\oscD{\ell}$.

If we do not use nodal interpolation to approximate $g\approx g_\ell$, the estimator reduction estimate~\eqref{eq:reduction} becomes
\begin{align}\label{eq*:reduction}
\varrho_{\ell+1}^2
\le q\,\varrho_\ell^2 + \c{reduction}\norm{U_{\ell+1}-U_\ell}{H^1(\Omega)}^2,
\end{align}
where $\c{reduction}>0$ additionally depends on $\Omega$. The reason for this is that the analysis provides an additional term $\norm{g_{\ell+1}-g_\ell}{H^{1/2}(\Gamma_D)}^2$ on the right-hand side of~\eqref{eq:reduction} since we loose the
orthogonality relation~\eqref{eq:nodal:orthogonality} which is used in the form
\begin{align*}
 \norm{h_{\ell+1}^{1/2}(g-g_{\ell+1})'}{L^2(\Gamma_D)}^2
 &\le \norm{h_{\ell+1}^{1/2}(g-g_{\ell+1})'}{L^2(\Gamma_D)}^2
 + \norm{h_{\ell+1}^{1/2}(g_{\ell+1}-g_{\ell})'}{L^2(\Gamma_D)}^2\\
 &= \norm{h_{\ell+1}^{1/2}(g-g_{\ell})'}{L^2(\Gamma_D)}^2.
\end{align*}
Instead, an inverse estimate and the Rellich compactness theorem yield
\begin{align*}
 \norm{\nabla(U_{\ell+1}-U_\ell)}{L^2(\Omega)}^2 + \norm{h_{\ell}^{1/2}(g_{\ell+1}-g_\ell)'}{L^2(\Gamma_D)}^2
 &\lesssim
 \norm{\nabla(U_{\ell+1}-U_\ell)}{L^2(\Omega)}^2 + \norm{g_{\ell+1}-g_\ell}{H^{1/2}(\Gamma_D)}^2\\
 &\simeq \norm{U_{\ell+1}-U_\ell}{H^1(\Omega)}^2
\end{align*}
which proves~\eqref{eq*:reduction}.
Note that this estimate holds for \emph{any} discretization of $g\approx g_\ell\in\SS^1(\EE_\ell^{D})$ and even in 3D, where the arclength derivative $(\cdot)'$ is replaced by the surface gradient $\nabla_\Gamma(\cdot)$; we refer to~\cite{ghs} for the inverse estimate.%

A possible choice for $g_\ell$ is $g_\ell = \Pi_\ell g$, where $\Pi_\ell:L^2(\Gamma_D)\to\SS^1(\EE_\ell^{D})$ is the 
$L^2$-orthogonal projection~\cite{bcd}. \new{Alternatively, $g_\ell = P_\ell g$, with $P_\ell: H^{1/2} \rightarrow \SS^1(\EE_\ell^{D})$ the
Scott-Zhang projection is chosen~\cite{sv}.}
Note that newest vertex bisection of $\TT_\ell$ and hence of $\EE_\ell^D$ ensures that $\Pi_\ell$ is a stable projection with respect to the 
$H^1(\Gamma_D)$-norm~\cite{kpp}. In~\cite{kp}, we prove \new{for either choice} the approximation estimate
\begin{align}
 \norm{g-g_\ell}{H^{1/2}(\Gamma_D)}
 \lesssim \norm{h_\ell^{1/2}\nabla_\Gamma(g-g_\ell)}{L^2(\Gamma_D)}
 =: \oscD{\ell}.
\end{align}
Moreover, we show that, for $g_\ell=\Pi_\ell g$, the a~priori limit $g_\infty:=\lim_\ell g_\ell$ exists strongly in $H^\alpha(\Gamma_D)$
for $0\le\alpha<1$ and even weakly in $H^1(\Gamma_D)$
provided that the discrete spaces $\SS^1(\EE_\ell^{D})$ are nested, i.e.\ $\SS^1(\EE_\ell^{D})\subseteq\SS^1(\TT_{\ell+1}|_{\Gamma_D})$ for all $\ell\in\N_0$. Note, however, that this is always the case for adaptive mesh-refining algorithms. In particular, we have
\begin{align}\label{eq:nestedness}
 \SS^1(\TT_\ell) \subseteq \SS^1(\TT_{\ell+1})
 \quad\text{for all }\ell\in\N_0.
\end{align}
In the following, we even aim to prove that nestedness~\eqref{eq:nestedness} implies the existence of the a~priori limit $\lim_\ell U_\ell$ in $H^1(\Omega)$. To that end, we need the following lemma.%

\begin{lemma}[a~priori convergence of Scott-Zhang projection]
\label{lem:apriori:sz} We recall the Scott-Zhang projection $P_\ell$ \new{onto $\SS^1(\TT_\ell)$} and make the additional
assumption that the edges $E_z$ are chosen appropriately, i.e.\ for $\omega_{\ell,z} \subset \bigcup(\TT_\ell \cap \TT_{\ell+1})$ we ensure
that the edge $E_{z}$ is chosen for both operators $P_\ell$ and $P_{\ell+1}$.
Then, the Scott-Zhang interpolands $v_\ell:=P_\ell v\in\SS^1(\TT_\ell)$ of arbitrary
$v\in H^1(\Omega)$ converge to some a~priori limit in $H^1(\Omega)$, i.e.\
there holds
\begin{align}\label{eq:apriori:sz}
 \norm{P_\infty v-P_\ell v}{H^1(\Omega)}
 \xrightarrow{\ell\to\infty}0
\end{align}
for a certain element $P_\infty v\in\SS^1(\TT_\infty):= \overline{\bigcup_{\ell\in N}\SS^1(\TT_\ell)}$.
\end{lemma}

\begin{proof}
We follow the ideas from~\cite{msv} and define the following subsets of $\Omega$:
\begin{align*}
\Omega_\ell^0 &:= \textstyle{\bigcup\{T \in \TT_\ell\,:\, \omega_{\ell}(T) \subset \bigcup\big(\bigcup_{i=0}^\infty \bigcap_{j=i}^\infty \TT_j\}}\big),\\
\Omega_{\ell} &:= \textstyle{\bigcup\{T \in \TT_\ell\,:\, \text{There exists }k \geq 0\text{ s.t. } \omega_{\ell}(T)\text{ is at least uniformly refined in }\TT_{\ell+k} \}},\\
\Omega_{\ell}^* &:= \Omega\setminus(\Omega_{\ell} \cup \Omega_\ell^0),
\end{align*}
where $\omega_{\ell}(\omega):=\bigcup\{T \in \TT_\ell \,:\, T\cap \omega \neq \emptyset\}$ for all measurable $\omega \subset \Omega$.
According to~\cite[Corollary 4.1]{msv}, it holds that
\begin{align}
\label{eq1:apriorisz}
 \lim_{\ell\to\infty}\norm{\chi_{\Omega_\ell}h_\ell}{L^\infty(\Omega)}=0.
\end{align}
Let $\eps>0$ be arbitrary. Since the space $H^2(\Omega)$ is dense in $H^1(\Omega)$, we find $v_\eps \in H^2(\Omega)$ such that
$\norm{v-v_\eps}{H^1(\Omega)} \leq \eps$. Due to local approximation and stability properties of $P_\ell$, we obtain
\begin{align*}
\norm{(1-P_\ell)v}{H^1(\Omega_{\ell})}\lesssim \norm{(1-P_\ell)v_\eps}{H^1(\Omega_{\ell})}+\eps\leq \norm{h_\ell\, D^2 v_\eps}{L^2(\omega_\ell(\Omega_{\ell}))}
+\eps,
\end{align*}
cf.~\cite{sz}.
By use of~\eqref{eq1:apriorisz}, we may choose $\ell_0 \in \N$ sufficiently large
to guarantee
$\norm{h_\ell\,D^2 v_\eps}{L^2(\omega_\ell(\Omega_{\ell}))}\leq
\norm{h_\ell}{L^\infty(\omega_\ell(\Omega_{\ell}))}\norm{D^2 v_\eps}{L^2(\Omega)}\leq \eps$ for all $\ell \geq \ell_0$. Then, there holds
\begin{align}
 \label{eq:szunif}
 \norm{(1-P_\ell)v}{H^1(\Omega_{\ell})}\lesssim \eps \quad \text{for all }\ell \geq \ell_0.
\end{align}
There holds $\lim_{\ell\to\infty}|\Omega_\ell^*|=0$, cf.~\cite[Proposition 4.2]{msv}, and this provides the existence of $\ell_1 \in \N$ such that
\begin{align}
\label{eq:omstar}
 \norm{v}{H^1(\omega_\ell(\Omega_{\ell}^*))}\leq \eps\quad \text{for all } \ell \geq \ell_1
\end{align}
due to the non-concentration of Lebesgue functions.
With these preparations, we finally aim at proving that $P_\ell v$ is a Cauchy sequence in $H^1(\Omega)$. Therefore, let $\ell \geq \max\{\ell_0,\ell_1\}$ and $k\geq 0$ be arbitrary.
First, we use that for any $T \in \TT_\ell$, $(P_\ell v)|_T$ depends only on $v|_{\omega_\ell(T)}$. Then, by definition of $\Omega_\ell^0$ and
our assumption on the definition of $P_\ell$ and $P_{\ell+k}$ on $\TT_\ell \cap \TT_{\ell+k}$, we obtain
\begin{align}
 \label{eq1:szapriori}
\norm{P_\ell v- P_{\ell+k} v }{H^1(\Omega_\ell^0)} = 0.
\end{align}
Second, due to the local stability of $P_\ell$ and~\eqref{eq:omstar}, there holds
\begin{align}
 \label{eq2:szapriori}
\begin{array}{rcl}
\norm{P_\ell v - P_{\ell+k} v}{H^1(\Omega_{\ell}^*)} &\leq& \norm{P_\ell v}{H^1(\Omega_{\ell}^*)}+\norm{P_{\ell+k} v}{H^1(\Omega_{\ell}^*)}\\
&\lesssim& \norm{v}{H^1(\omega_\ell(\Omega_{\ell}^*))}+\norm{v}{H^1(\omega_{\ell+k}(\Omega_{\ell}^*))}\\
&\leq& 2 \norm{v}{H^1(\omega_\ell(\Omega_{\ell}^*))} \leq 2 \eps.
\end{array}
\end{align}
Third, we proceed by exploiting~\eqref{eq:szunif}. We have
\begin{align}
 \label{eq3:szapriori}
\norm{P_\ell v - P_{\ell+k} v}{H^1(\Omega_{\ell})} \leq \norm{P_\ell v - v}{H^1(\Omega_{\ell})}+\normHe{v-P_{\ell+k}v}{\Omega_{\ell}}
 \lesssim \eps.
\end{align}
Combining the estimates from~\eqref{eq1:szapriori}--\eqref{eq3:szapriori}, we conclude
 $\norm{P_\ell v - P_{\ell+k} v}{H^1(\Omega)} \lesssim \eps$,
i.e.\ $(P_\ell v)$ is a Cauchy sequence in $H^1(\Omega)$ and hence convergent.
\end{proof}

Now, we are able to prove a~priori convergence of $U_\ell$ towards some
a~priori limit $u_\infty$.

\begin{proposition}[a~priori convergence of $U_\ell$]\label{prop:apriori}
Suppose that the discrete spaces satisfy nestedness~\eqref{eq:nestedness} and that $U_\ell\in\SS^1(\TT_\ell)$ solves~\eqref{eq:galerkin} with $g_\ell=\Pi_\ell g$ and $\Pi_\ell:L^2(\Gamma_D)\to\SS^1(\EE_\ell^{D})$ the $L^2$-projection. Then, the a~priori limit $u_\infty:=\lim\limits_{\ell\to\infty}U_\ell\in H^1(\Omega)$
exists.
\end{proposition}

\begin{proof}
For $g_\ell\in H^{1/2}(\Gamma)$, we consider the continuous auxiliary problem
\begin{align*}
 -\Delta w_\ell&=0 \quad\text{in }\Omega,\\
 w_\ell &= g_\ell\quad\text{on }\Gamma_D,\\
 \partial_nw_\ell &=0 \quad\text{on }\Gamma_N.
\end{align*}
Let $w_\ell\in H^1(\Omega)$ be the unique (weak) solution and note that the trace $\widehat g_\ell:=w_\ell|_\Gamma \in H^{1/2}(\Gamma)$ provides an extension of $g_\ell$ with
\begin{align*}
 \norm{\widehat g_\ell}{H^{1/2}(\Gamma)}
 \le \norm{w_\ell}{H^1(\Omega)}
 \lesssim \norm{g_\ell}{H^{1/2}(\Gamma_D)}
 \le \norm{\widehat g_\ell}{H^{1/2}(\Gamma)}.
\end{align*}
For arbitrary $k,\ell\in\N$, the same type of arguments proves
\begin{align*}
 \norm{\widehat g_\ell-\widehat g_k}{H^{1/2}(\Gamma)}
 \simeq \norm{g_\ell-g_k}{H^{1/2}(\Gamma_D)}.
\end{align*}
Since $(g_\ell)$ is a Cauchy sequence in $H^{1/2}(\Gamma_D)$, cf.~\cite{kp}, we obtain that $(\widehat g_\ell)$ is a Cauchy sequence in $H^{1/2}(\Gamma)$, whence convergent with limit $\widehat g_\infty\in H^{1/2}(\Gamma)$.

Second, note that $(\LL_\ell\widehat g_\ell)|_{\Gamma_D} = g_\ell$, where $\LL_\ell=P_\ell\LL$ denotes the discrete lifting from~\eqref{eq:discreteLifting}.
Therefore, $\widetilde U_\ell := U_\ell - \LL_\ell \widehat g_\ell \in \SS^1_D(\TT_\ell)$ is the unique solution of the variational form
\begin{align}
\label{eq1:weak3d}
 \dual{\nabla\widetilde U_\ell}{\nabla V_\ell}_\Omega=\dual{\nabla u}{\nabla V_\ell}_\Omega-\dual{\nabla\LL_\ell \widehat g_\ell}{\nabla V_\ell}_\Omega \quad \text{ for all } V_\ell \in \SS^1_D(\TT_\ell).
\end{align}
Third, Lemma~\ref{lem:apriori:sz} implies
\begin{align*}
 \normHe{\LL_\ell \widehat g_\ell -P_\infty \LL \widehat g_\infty}{\Omega}
 &\leq \normHe{P_\ell(\LL \widehat g_\ell - \LL \widehat g_\infty)}{\Omega}
 + \normHe{P_\ell \LL \widehat g_\infty - P_\infty \LL \widehat g_\infty}{\Omega}\\
 &\lesssim\normHeh{\widehat g_\ell - \widehat g_\infty}{\Gamma}
 +\normHe{P_\ell \LL \widehat g_\infty - P_\infty \LL \widehat g_\infty}{\Omega}
 \xrightarrow{\ell\to\infty} 0.
\end{align*}
Fourth, let $\widetilde U_{\ell,\infty} \in \SS^1_D(\TT_\ell)$ be the unique solution of the discrete auxiliary problem
\begin{align}
\label{eq2:weak3d}
\dual{\nabla\widetilde U_{\ell,\infty}}{\nabla V_\ell}_\Omega=\dual{\nabla u}{\nabla V_\ell}_\Omega-\dual{\nabla P_\infty\LL \widehat g_\infty}{\nabla V_\ell}_\Omega
\quad \text{ for all } V_\ell \in \SS^1_D(\TT_\ell).
\end{align}
Due to the nestedness of the ansatz spaces $\SS^1_D(\TT_\ell)$, we derive a priori convergence $\widetilde U_{\ell,\infty} \xrightarrow{\ell\to\infty} \widetilde u_\infty\in H^1(\Omega)$, where $\widetilde u_\infty$ denotes the Galerkin solution with respect to the closure of $\bigcup^\infty_{\ell = 0} \SS^1_D(\TT_\ell)$ in $H^1_0(\Omega)$ , see e.g.~\cite[Lemma~6.1]{bv}. With the stability of~\eqref{eq1:weak3d} and~\eqref{eq2:weak3d}, we obtain
\begin{align*}
 \norm{\nabla(\widetilde U_{\ell,\infty} - \widetilde U_{\ell})}{L^2(\Omega)} \lesssim \normHe{\LL_\ell \widehat g_\ell- P_\infty \LL \widehat g_\infty}{\Omega} \xrightarrow{\ell\to\infty} 0,
\end{align*}
and therefore $\widetilde U_\ell \xrightarrow{\ell\to\infty} \widetilde u_\infty$ in $H^1(\Omega)$.
Finally, we conclude
\begin{align*}
 U_\ell = \widetilde U_\ell + \LL_\ell \widehat g_\ell \xrightarrow{\ell\to\infty} \widetilde u_\infty + P_\infty \LL \widehat g_\infty =:u_\infty \in H^1(\Omega),
\end{align*}
which concludes the proof.
\end{proof}

\new{
\begin{remark}
 Note that Proposition~\ref{prop:apriori} also holds if the Scott-Zhang projection is used to discretize $g \approx g_\ell = P_\ell g$.
This immediately follows from Lemma~\ref{lem:apriori:sz}, since $g_\ell = (P_\ell \LL g)|_{\Gamma_D} \to (P_\infty \LL g)|_{\Gamma_D}$ as $\ell\to\infty$.
\end{remark}
}

\new{
\begin{theorem}\label{thm:3Dconvergence}
Suppose that either the $L^2$-projection $g_\ell = \Pi_\ell g$ or the Scott-Zhang operator $g_\ell = P_\ell g$
is used to discretize the Dirichlet data $g\in H^1(\Gamma)$. Then, Algorithm~\ref{algorithm:doerfler}
guarantees $\lim_\ell \norm{u - U_\ell}{H^1(\Omega)} = 0$ for both 2D and 3D.
\end{theorem}
}

\new{
\begin{proof}
With Proposition~\ref{prop:apriori} and the estimator reduction~\eqref{eq*:reduction}, we obtain
\begin{align*}
 \varrho_{\ell+1}^2 \le q\,\varrho_\ell^2 + \alpha_\ell,
 \quad\text{where}\quad 0<q<1
 \text{ and }\alpha_\ell\ge0
 \text{ with }\alpha_\ell\xrightarrow{\ell\to\infty}0.
\end{align*}
From this and elementary calculus, we deduce estimator convergence $\lim_\ell\varrho_\ell = 0$, 
cf.~\cite{afp} for the concept of estimator reduction.
According to reliability of $\varrho_\ell$, this yields convergence of the adaptive algorithm. 
\end{proof}
}

Note, however, that this convergence result is much weaker than the contraction result of
Theorem~\ref{thm:contraction}. With the techniques of the present paper, it is unclear how to prove a contraction result if the
additional orthogonality relation~\eqref{eq:nodal:orthogonality} fails to hold.

\begin{figure}[h]
\begin{center}%
\includegraphics[width=40mm]{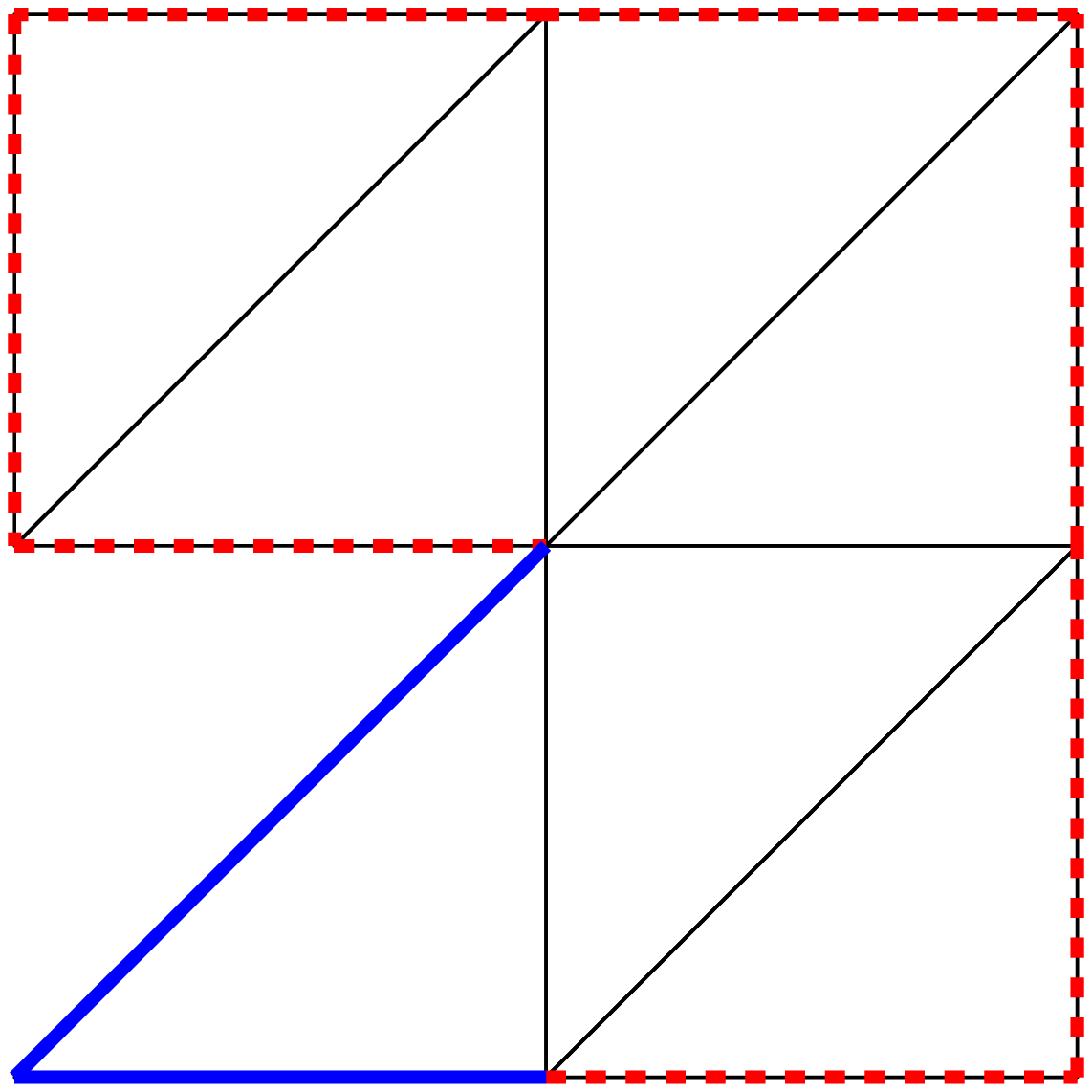}
\hspace{20mm}
\includegraphics[width=40mm]{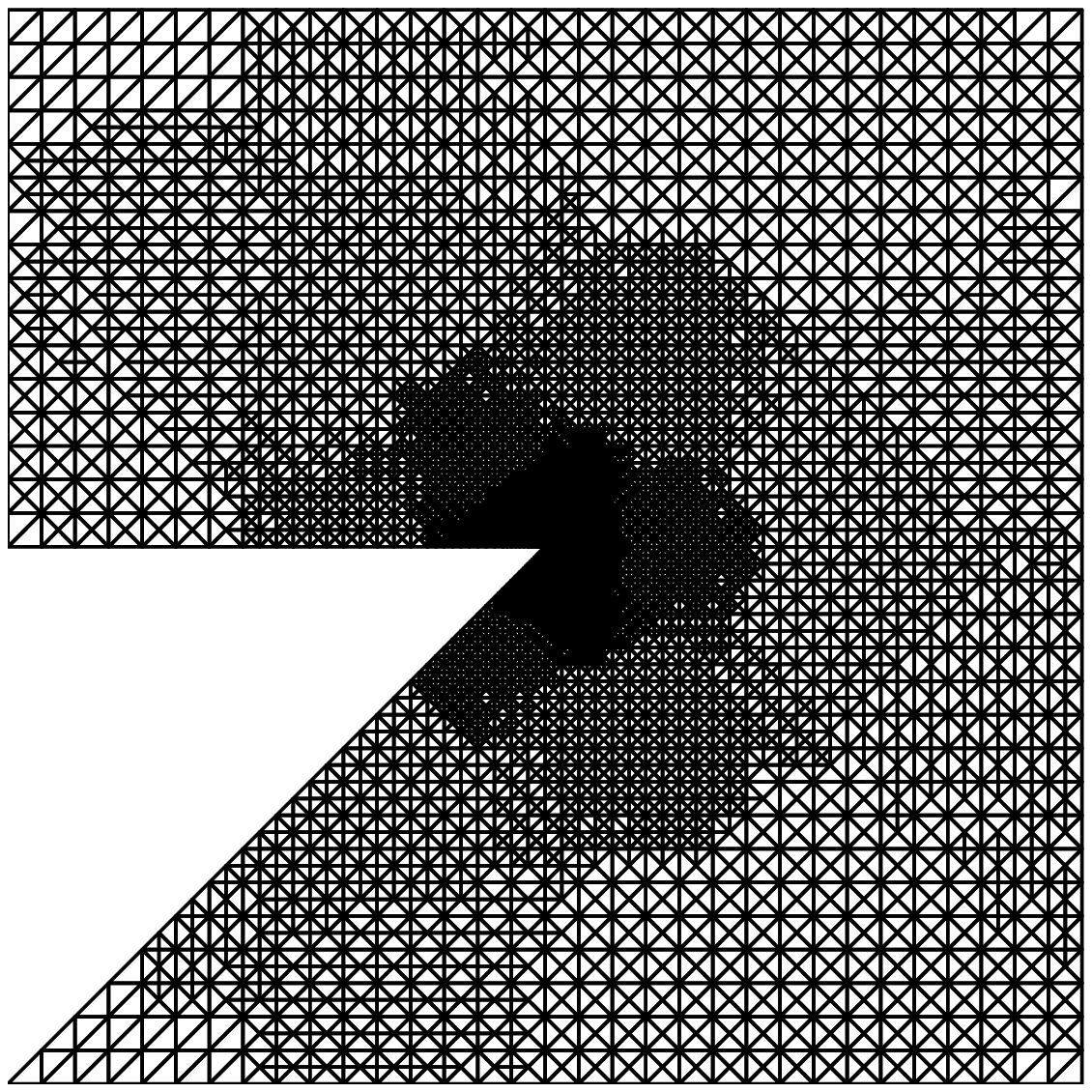}
\caption{Z-shaped domain with initial mesh $\TT_0$ and adaptively generated mesh $\TT_9$ with $N = 10966$ for $\theta = 0.5$ in Algorithm~\ref{algorithm:doerfler}. The Dirichlet boundary $\Gamma_D$ is marked with a solid line, whereas the dashed line denotes the Neumann boundary $\Gamma\backslash\Gamma_D$.}
\label{fig:meshes}
\end{center}
\end{figure}
\begin{figure}[t]
\begin{center}
\psfrag{aneumannOsc}{$\eta_{N,\ell}$ (adap.)}
\psfrag{ajumpsN}{$\eta_{\Omega,\ell}$ (adap.)}
\psfrag{ujumpsN}{$\eta_{\Omega,\ell}$ (unif.)}
\psfrag{uneumannOsc}{$\eta_{N,\ell}$ (unif.)}
\psfrag{adirOsc}{$\oscD\ell$ (adap.)}
\psfrag{udirOsc}{$\oscD\ell$ (unif.)}
\psfrag{aedgeOsc}{$\osc\ell$ (adap.)}
\psfrag{uedgeOsc}{$\osc\ell$ (unif.)}
\psfrag{O12}[cc][cc][1][-15]{$\mathcal{O}(N^{-1/2})$}
\psfrag{O27}[cc][cc][1][-10]{$\mathcal{O}(N^{-2/7})$}
\psfrag{O34}[cc][cc][1][-23]{$\mathcal{O}(N^{-3/4})$}
\includegraphics[width=170mm]{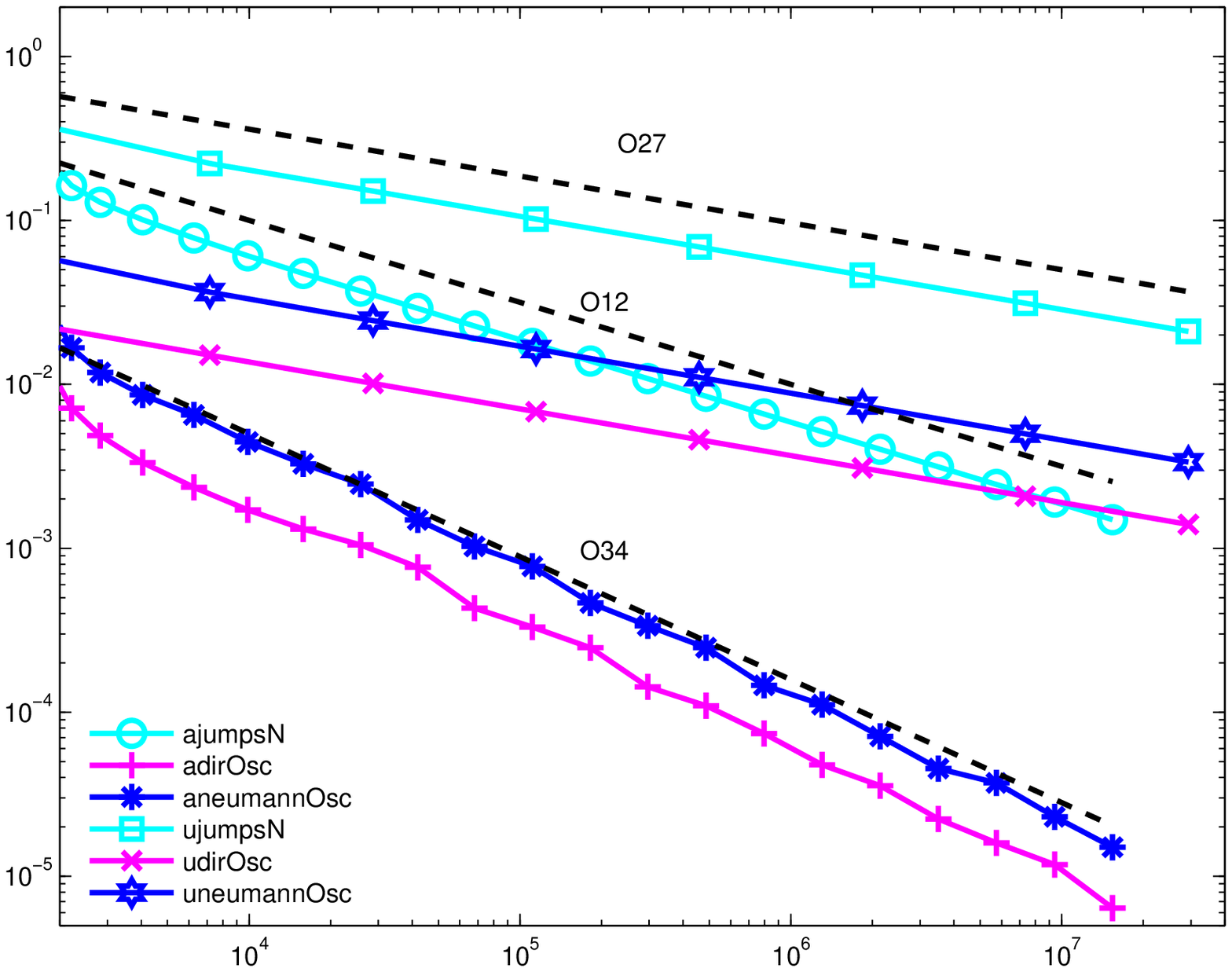}
\caption{Numerical results for $\varrho_\ell$ for uniform and adaptive mesh-refinement with Algorithm~\ref{algorithm:doerfler} 
resp.\ \new{the modified D\"orfler marking} and $\theta \in \{0.2, 0.5, 0.8\}$,
plotted over the number of ele\-ments $N = \# \TT_\ell$.}
\label{fig:convRho}
\end{center}
\end{figure}
\begin{figure}[t]
\begin{center}
\psfrag{t02}{$\theta = 0.2$ (Alg.~\ref{algorithm:doerfler})}
\psfrag{t05}{$\theta = 0.5$ (Alg.~\ref{algorithm:doerfler})}
\psfrag{t08}{$\theta = 0.8$ (Alg.~\ref{algorithm:doerfler})}
\psfrag{uniform}{uniform}
\psfrag{t02m}{$\theta = 0.2$ (mod.)}
\psfrag{t05m}{$\theta = 0.5$ (mod.)}
\psfrag{t08m}{$\theta = 0.8$ (mod.)}
\psfrag{O12}[cc][cc][1][-25]{$\mathcal{O}(N^{-1/2})$}
\psfrag{O27}[tc][bc][1][-15]{$\mathcal{O}(N^{-2/7})$}
\psfrag{O34}[cc][cc][1][-23]{$\mathcal{O}(N^{-3/4})$}
\includegraphics[width=170mm]{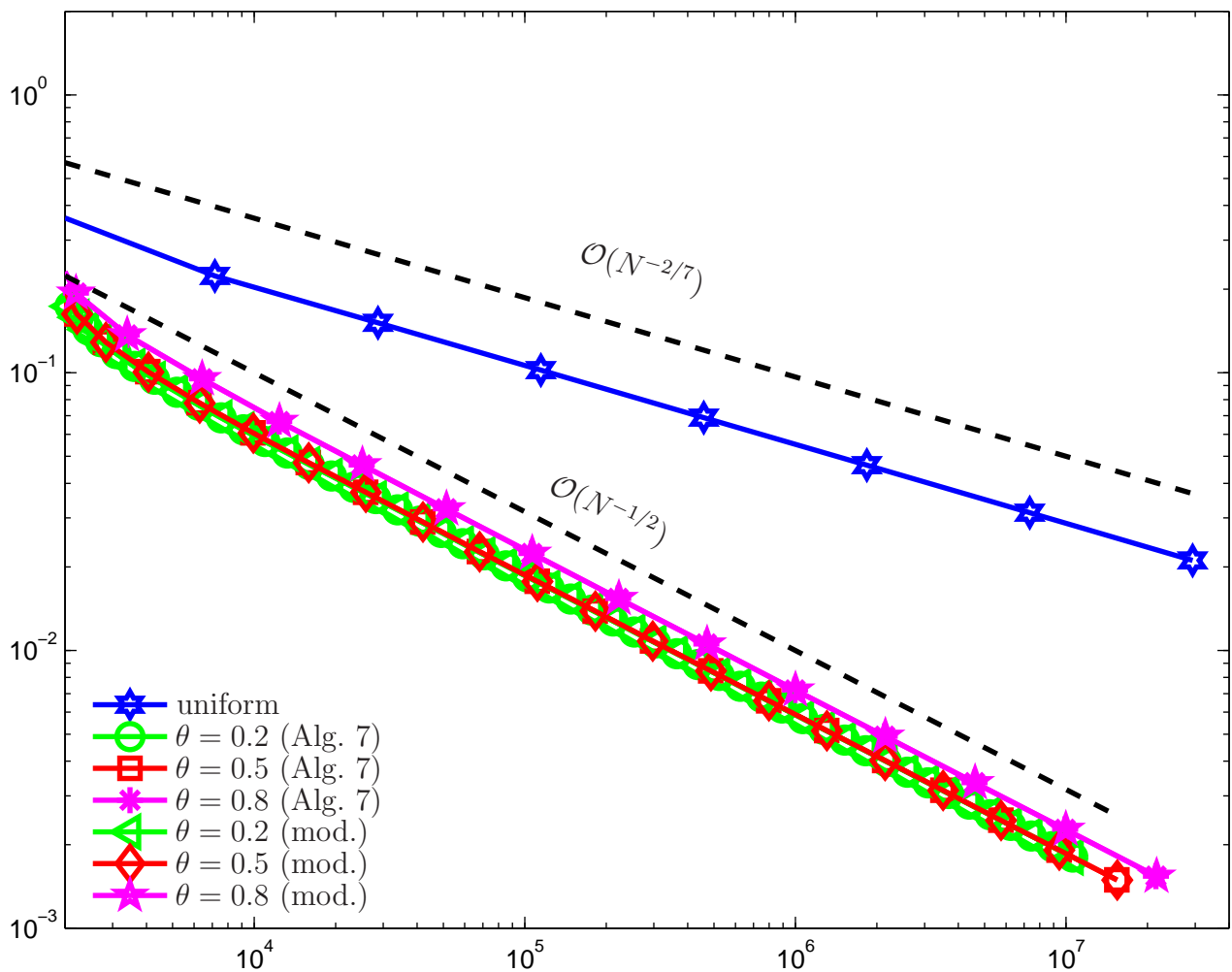}
\caption{Numerical results for $\eta_{\Omega,\ell}$, $\oscD\ell$, and $\eta_{N,\ell}$ for uniform and adaptive mesh-refinement with
Algorithm~\ref{algorithm:doerfler} and $\theta = 0.5$, plotted over the number of ele\-ments $N = \# \TT_\ell$. Adaptive refinement leads to optimal
convergence rates.}
\label{fig:convTotal}
\end{center}
\end{figure}
%
\section{Numerical Experiment}
\label{section:numerics}%
\subsection{Example with known solution}
\noindent
On the Z-shaped domain $\Omega=(-1, 1)^2\backslash \conv\{(0,0),$\linebreak$ (-1,-1), (0,-1)\}$, we consider the mixed boun\-dary value 
problem~\eqref{eq:strongform}, where the partition of the boundary $\Gamma=\partial\Omega$ into Dirichlet boundary $\Gamma_D$ and 
Neumann boundary $\Gamma_N$ as well as the initial mesh are shown in Figure~\ref{fig:meshes}. We prescribe the exact solution $u(x)$ 
in polar coordinates by
\begin{align}
 u(x) = r^{4/7}\cos(4\varphi/7)
 \quad\text{for }x=r\,(\cos\varphi,\sin\varphi).
\end{align}
Then, $f=-\Delta u\equiv0$, and the solution $u$ as well as its Dirichlet data $g = u|_{\Gamma_D}$ admit a generic singularity at the 
reentrant corner $r = 0$. 

Figure~\ref{fig:convRho} shows a comparison between uniform and adaptive mesh refinement. For \new{the algorithm} based on the 
modified D\"orfler marking, we use $\theta:= \vartheta = \theta_1 = \theta_2$. For \new{both algorithms}, we then vary the adaptivity parameter $\theta$
between $0.2$ and $0.8$. We observe that both adaptive algorithms lead to the optimal convergence rate $\mathcal{O}(N^{-1/2})$
for all choices of $\theta$, whereas uniform refinement leads only to suboptimal convergence behaviour of approximately $\mathcal{O}(N^{-2/7})$.

Note that due to $f \equiv 0$, we have $\osc\ell \equiv 0$ in this example.
In Figure~\ref{fig:convTotal}, we compare the jump terms
\[
\eta_{\Omega,\ell}^2 := \sum_{E\in\EE_\ell^\Omega}|E|\norm{[\partial_nU_\ell]}{L^2(E)}^2,
\]
the Dirichlet data oscillations $\oscD\ell$, and the Neumann jump terms
\[
\eta_{N,\ell}^2 := \sum_{E \in \EE_\ell^N} |E|\norm{\phi - \partial_nU_\ell}{L^2(E)}^2
\]
for uniform and adaptive refinement.
Due to the corner singularity at $r = 0$, uniform refinement leads to a suboptimal convergence behaviour for $\eta_{\Omega,\ell}$ and even for $\oscD\ell$ and $\eta_{N,\ell}$, i.e.\ all contributions of $\varrho_\ell^2 = \eta_{\Omega,\ell}^2 + \eta_{N,\ell}^2 + \oscD{\ell}$ show the same poor convergence rate of approximately $\OO(N^{-2/7})$. For adaptive mesh-refinement, we observe that the optimal order of convergence is retained, namely $\varrho_\ell \simeq \eta_\ell = \OO(N^{-1/2})$. Moreover, we even observe optimal convergence behaviour $\oscD{\ell} \simeq \eta_{N,\ell} = \OO(N^{-3/4})$ for the boundary contributions of $\varrho_\ell$.

Finally, in Figure~\ref{fig:meshes}, the initital mesh $\TT_0$ and the adaptively generated mesh $\TT_9$ with $N=10966$ Elements
are visualized. As expected, adaptive refinement is essentially concentrated around the reentrant corner $r = 0$.
\subsection{Example with unknown solution}
On the L-shaped domain $\Omega=(-1,1)^2\setminus (-1,0)\times (0,1)$, we consider the mixed boundary value problem~\eqref{eq:strongform}. The initial configuration with Dirichlet boundary $\Gamma_D$, Neumann boundary $\Gamma_N$, as well as the initial mesh is shown in Figure~\ref{fig:meshes2}. 
For the unknown solution $u\in H^1(\Omega)$, we prescribe in polar coordinates with respect to $(0,0)$
\begin{align*}
g &= u|_{\Gamma_D} = r^{2/3}\sin(2\varphi/3)\quad\text{on }\Gamma_D,\\
\phi&= \partial_n u = 0\qquad\quad\quad\;\,\qquad\text{on }\Gamma_N,\\
f&=-\Delta u = |1-r|^{-1/4}\quad\;\,\,\text{in }\Omega.
\end{align*}
There holds $g\in H^1(\Gamma_D)$, $\phi\in L^2(\Gamma_N)$, and $f\in L^2(\Omega)$. Note that the Dirichlet data $g$ has a singularity at the reentrant corner $(0,0)$, whereas the volume force $f$ is singular along the circle around $(0,0)$ with radius $r=1$. 
Again, we compare the standard D\"orfler marking strategy as well the modified D\"orfler marking with the uniform approach.
\begin{figure}
\begin{center}
\includegraphics[width=40mm]{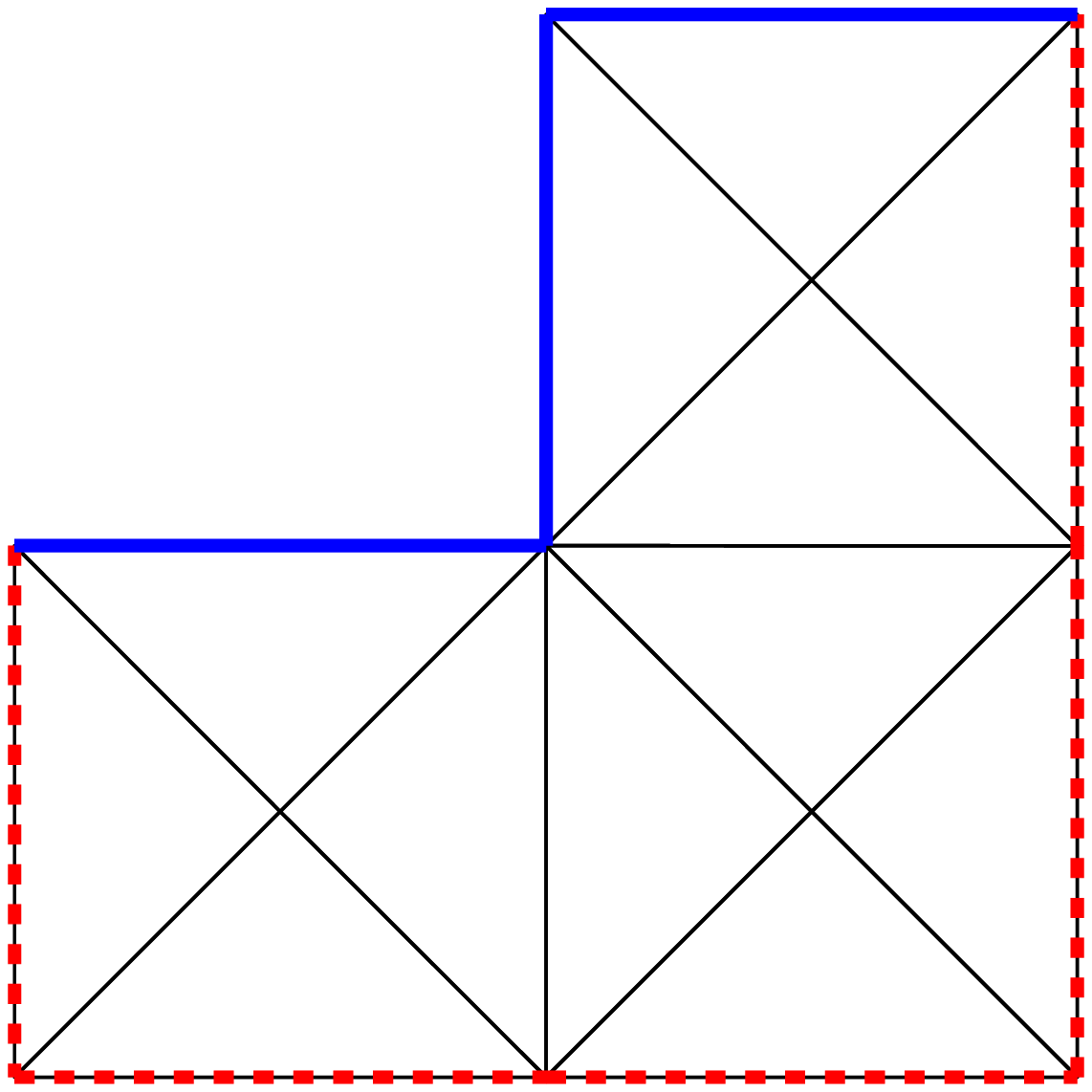}
\hspace{20mm}
\includegraphics[width=40mm]{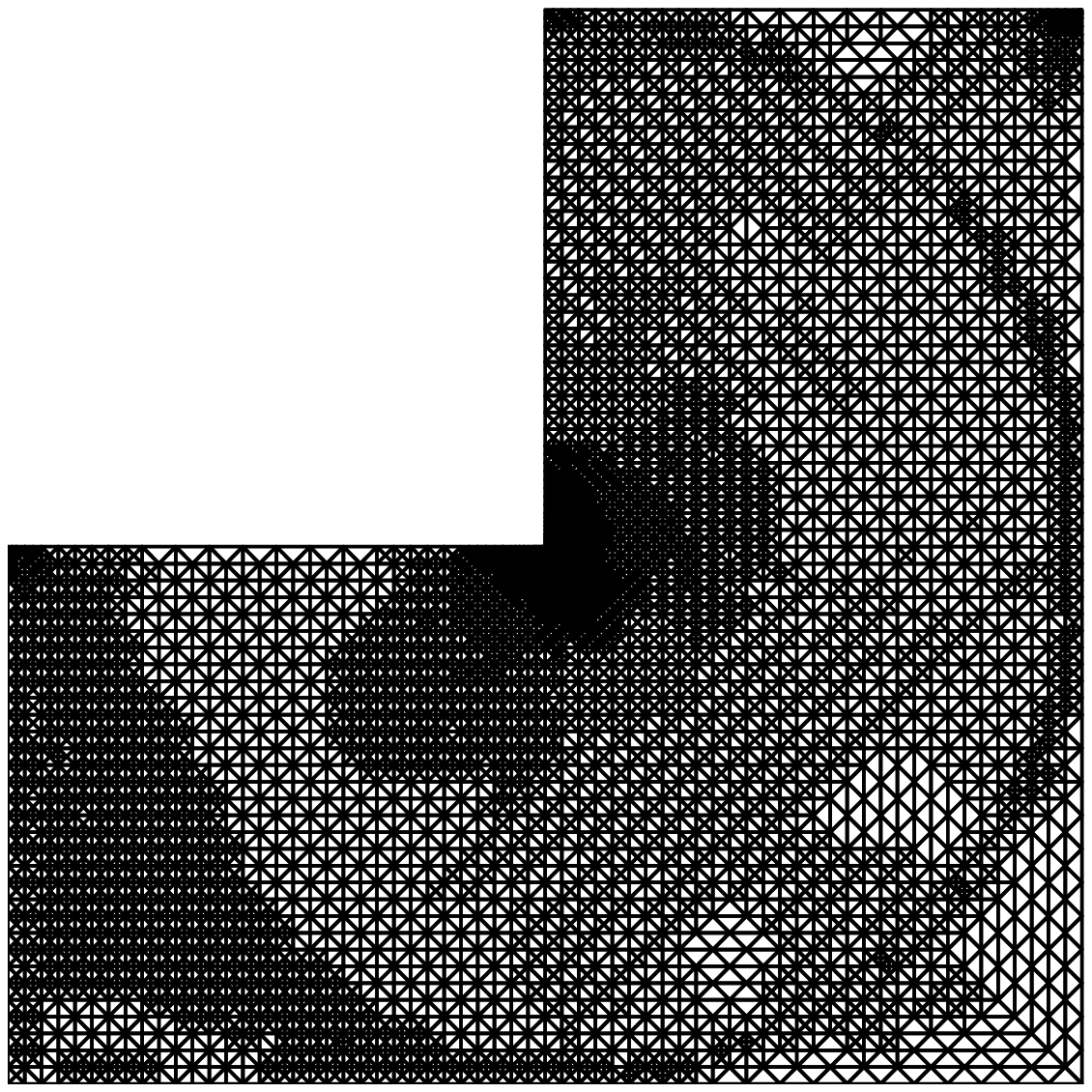}
\caption{L-shaped domain with initial mesh $\TT_0$ and adaptively generated mesh $\TT_9$ with $N = 12177$ for $\theta = 0.5$ in Algorithm~\ref{algorithm:doerfler}. The Dirichlet boundary $\Gamma_D$ is marked with a solid line, whereas the dashed line denotes the Neumann boundary $\Gamma\backslash\Gamma_D$.}
\label{fig:meshes2}
\end{center}
\end{figure}
Figure~\ref{fig:convRho2} shows a comparison between uniform and adaptive mesh refinement. The parameters $\theta=\vartheta=\theta_1=\theta_2$ are varied between $0.2$ and $0.8$. Both adaptive algorithms lead to optimal convergence rate $\mathcal{O}(N^{-1/2})$ for all choices of $\theta$, whereas uniform refinement leads only to a suboptimal rate of $\mathcal{O}(N^{-1/3})$.
\begin{figure}
\begin{center}
\psfrag{t02}{$\theta = 0.2$ (Alg.~\ref{algorithm:doerfler})}
\psfrag{t05}{$\theta = 0.5$ (Alg.~\ref{algorithm:doerfler})}
\psfrag{t08}{$\theta = 0.8$ (Alg.~\ref{algorithm:doerfler})}
\psfrag{uniform}{uniform}
\psfrag{t02m}{$\theta = 0.2$ (mod.)}
\psfrag{t05m}{$\theta = 0.5$ (mod.)}
\psfrag{t08m}{$\theta = 0.8$ (mod.)}
\psfrag{O12}[tc][bc][1][-25]{$\mathcal{O}(N^{-1/2})$}
\psfrag{O13}[cc][cc][1][-15]{$\mathcal{O}(N^{-1/3})$}
\psfrag{O34}{$\hspace{-1.2cm}\mathcal{O}(N^{-3/4})$}
\includegraphics[width=140mm]{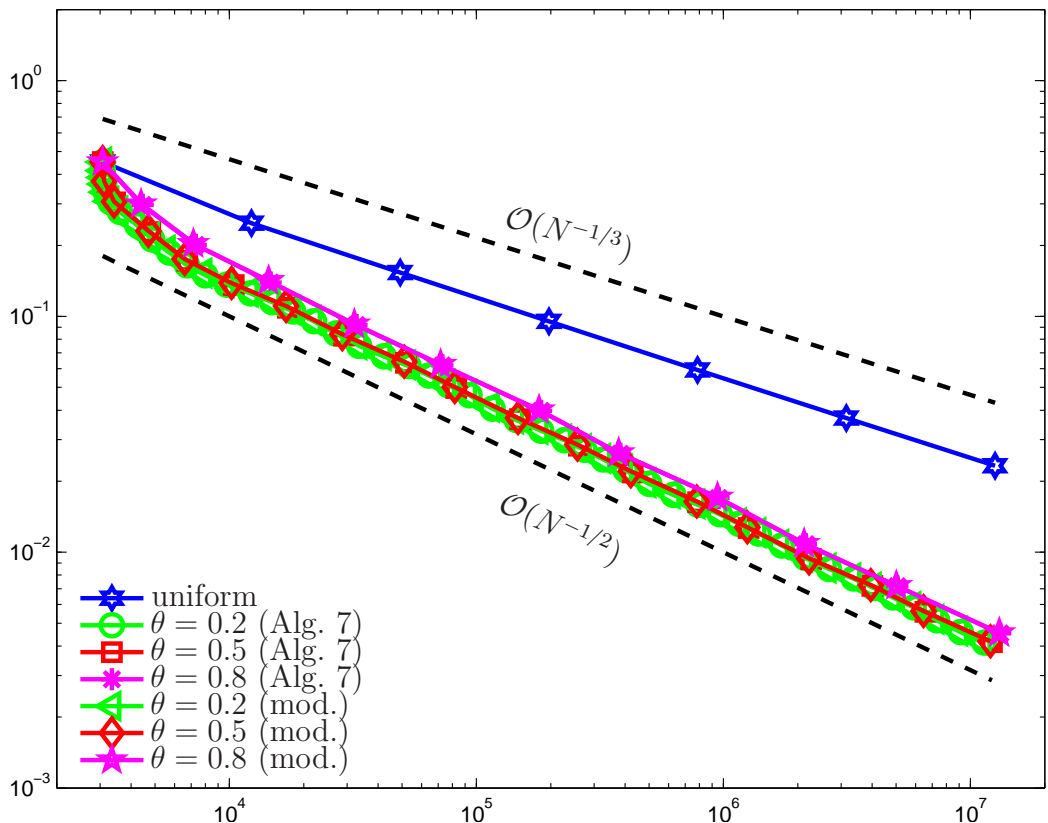}
\caption{Numerical results for $\varrho_\ell$ for uniform and adaptive mesh-refinement with Algorithm~\ref{algorithm:doerfler} 
resp.\ \new{the modified D\"orfler marking} and $\theta \in \{0.2, 0.5, 0.8\}$,
plotted over the number of ele\-ments $N = \# \TT_\ell$.}
\label{fig:convRho2}
\end{center}
\end{figure}
In Figure~\ref{fig:convTotal2}, we compare the estimator contributions which (in contrast to the previous example) include additional volume oscillations $\osc\ell$. Due to the data singularities, as well as the singularity introduced by the change of the boundary condition, uniform refinement leads only to suboptimal convergence rates for all estimator contributions. For adaptive mesh-refinement, we observe that the optimal order of convergence is retained. This means $\varrho_\ell \simeq \eta_\ell = \OO(N^{-1/2})$ and includes even optimal convergence behaviour $\oscD{\ell} \simeq \eta_{N,\ell} = \OO(N^{-3/4})$ for the boundary contributions of $\varrho_\ell$. 
In Figure~\ref{fig:meshes2}, one observes the adaptive refinement towards the singularity in the reentrant
corner as well as the circular singularity of $f$ and the singularities which stem from the change of boundary conditions.
\begin{figure}
\begin{center}
\psfrag{aneumannOsc}{$\eta_{N,\ell}$ (adap.)}
\psfrag{ajumpsN}{$\eta_{\Omega,\ell}$ (adap.)}
\psfrag{ujumpsN}{$\eta_{\Omega,\ell}$ (unif.)}
\psfrag{uneumannOsc}{$\eta_{N,\ell}$ (unif.)}
\psfrag{adirOsc}{$\oscD\ell$ (adap.)}
\psfrag{udirOsc}{$\oscD\ell$ (unif.)}
\psfrag{aedgeOsc}{$\osc\ell$ (adap.)}
\psfrag{uedgeOsc}{$\osc\ell$ (unif.)}
\psfrag{O12}[tc][bc][1][-17]{$\mathcal{O}(N^{-1/2})$}
\psfrag{O13}[cc][cc][1][-12]{$\mathcal{O}(N^{-1/3})$}
\psfrag{O34}[tc][bc][1][-25]{$\mathcal{O}(N^{-3/4})$}
\includegraphics[width=140mm]{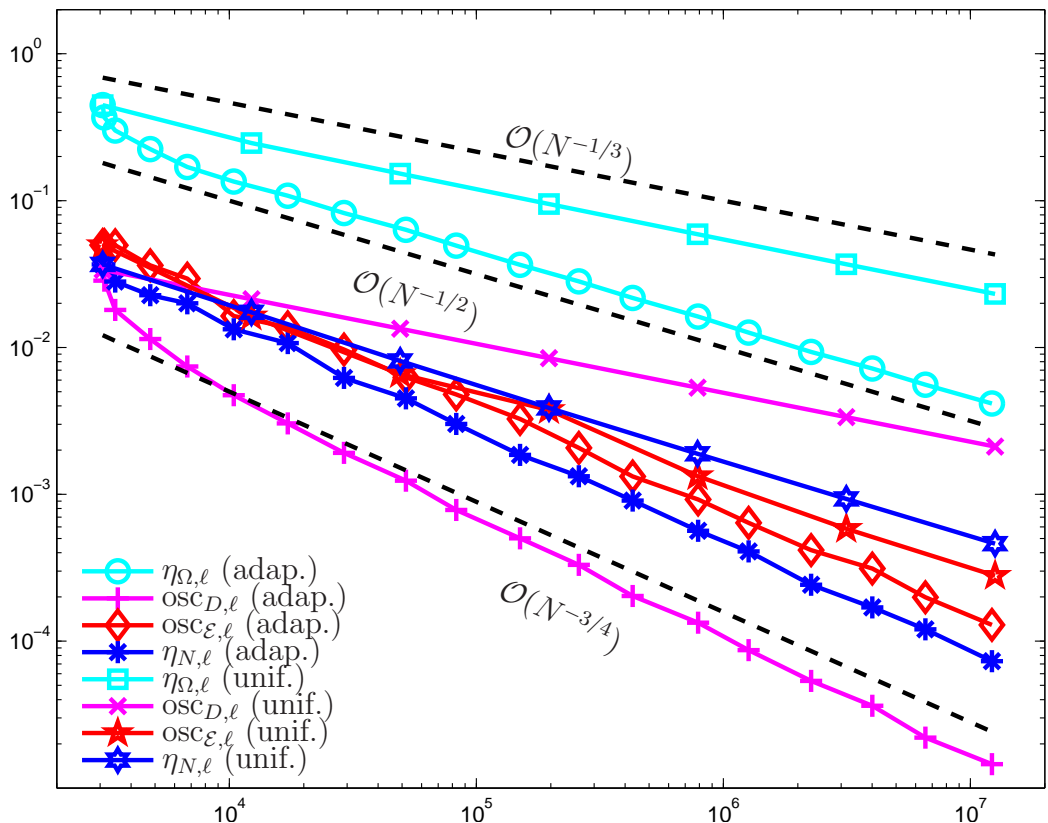}
\caption{Numerical results for $\eta_{\Omega,\ell}$, $\oscD\ell$, and $\eta_{N,\ell}$ for uniform and adaptive mesh-refinement with
Algorithm~\ref{algorithm:doerfler} and $\theta = 0.5$, plotted over the number of ele\-ments $N = \# \TT_\ell$. Adaptive refinement leads to optimal
convergence rates.}
\label{fig:convTotal2}
\end{center}
\end{figure}

\textbf{Acknowledgement.} The authors M.F. and D.P. are funded by the Austrian Science Fund (FWF) under grant P21732 \textit{Adaptive Boundary Element Method}, which is thankfully acknowledged. M.P. acknowledges support through the project \textit{Micromagnetic Simulations and Computational Design of Future Devices}, funded by the Viennese Science and Technology Fund (WWTF) under grant MA09-029
\newcommand{\bibentry}[2][!]{\ifthenelse{\equal{#1}{!}}{\bibitem{#2}}{\bibitem{#2}}}

\end{document}